\documentclass[11pt]{amsart}
\usepackage{bbm,amsfonts,latexsym,amsmath,amscd,amssymb,fancyhdr}
\usepackage{mathabx}

\usepackage{fullpage}

\usepackage{color}

\usepackage[all]{xy}

\theoremstyle{plain}

\newtheorem{theorem}{Theorem}
\numberwithin{theorem}{section}

\newtheorem{corollary}{Corollary}
\numberwithin{corollary}{section}

\newtheorem{definition}{Definition}
\numberwithin{definition}{section}

\newtheorem{lemma}{Lemma}
\numberwithin{lemma}{section}

\numberwithin{proposition}{section}

\newtheorem{remark}{Remark}
\numberwithin{remark}{section}

\numberwithin{example}{section}

\numberwithin{equation}{section}

\newcommand {\be}{\begin{equation}}
\newcommand {\ee}{\end{equation}}

\newcommand{\h}{\begin{eqnarray*}}
\newcommand{\e}{\end{eqnarray*}}

\newcommand{\CC}{\mathbb{C}}
\newcommand{\RR}{\mathbb{R}}
\newcommand{\ZZ}{\mathbb{Z}}
\newcommand{\QQ}{\mathbb{Q}}

\newcommand{\ch}{\mathrm{ch}}

\newcommand{\ii}{\sqrt{-1}}

\newcommand{\ind}{\mathrm{Ind}}

\newcommand{\reta}{\overline{\eta}}

\newcommand{\LL}{\mathcal{L}}

\begin{document}

\title[Cubic forms, anomaly cancellation and modularity]{Cubic forms, anomaly cancellation and modularity}
\author{Fei Han}
\address{Fei Han, Department of Mathematics, National University of Singapore,
 Block S17, 10 Lower Kent Ridge Road,
Singapore 119076 (mathanf@nus.edu.sg)}
 \author{Ruizhi Huang}
\address{Ruizhi Huang, Institute of Mathematics and Systems Sciences, Chinese Academy of Sciences, Beijing 100190, China}
\email{huangrz@amss.ac.cn}  
\urladdr{https://sites.google.com/site/hrzsea}
\author{Kefeng Liu}
\address{Kefeng Liu, Department of Mathematics, University of California at Los Angeles,
Los Angeles, CA 90095, USA (liu@math.ucla.edu) and Mathematical Science Research Center, Chongqing University of Technology, Chongqing, 400054, China}
\author{Weiping Zhang}
\address{Weiping Zhang, Chern Institute of Mathematics \& LPMC, Nankai
University, Tianjin 300071, China. (weiping@nankai.edu.cn)}
\maketitle

\begin{abstract} 
Recently Freed and Hopkins \cite{FH19} proved that there is no parity anomaly in M-theory on pin$^+$ manifolds in the low-energy field theory approximation, and they also developed an algebraic theory of cubic forms.
Earlier Witten \cite{W97} proved the anomaly cancellation for spin manifolds by introducing the $E_8$-bundle technique.
Motivated by the cubic forms and the anomaly cancellation formulas of Witten-Freed-Hopkins, we give some new cubic forms on spin, spin$^c$, spin$^{\omega_2}$ and orientable 12-manifolds respectively. We relate them to $\eta$-invariants when the manifolds are with boundary, and mod 2 indices on 10 dimensional characteristic submanifolds when the manifolds are spin$^c$ or spin$^{\omega_2}$. Our method of producing these cubic forms is a combination of (generalized) Witten classes and the character of the basic representation of affine $E_8$.

\end{abstract} 

\tableofcontents

\section*{Introduction} 

\subsection{Background} \label{background}

M-theory is an 11-dimensional theory that unifies string theories and is the best candidate for a theory of quantum gravity \cite{W95, Town96}. The theory has very rich structures both in physics and mathematics. In particular, the anomalies and their cancellations in the theory are related to profound aspects in geometry and topology. In the recent paper \cite{FH19}, Freed and Hopkins proved that there is no parity anomaly in M-theory on pin$^+$ manifolds (manifolds not necessarily orientable with the second Stiefel-Whitney class $\omega_2=0$) in the low-energy field theory approximation, which shows the consistency of the time-reversal symmetric theory. More precisely, Freed and Hopkins proved that the anomaly arising from the Rarita-Schwinger field and the anomaly arising from the ``Chern-Simons coupling" of the $C$-field cancel. Earlier, the anomaly cancellation for spin manifolds case was discovered by Witten \cite{W97}. 

We first briefly recap the Freed-Hopkins' anomaly cancellation in the pin$^+$ case \cite{FH19} and Witten's anomaly cancellation in the spin case \cite{W97}. Anomaly of an 11-dimensional theory is an invertible 12-dimensional theory. Let $W$ be a 12-dimensional pin$^+$ manifold. In \cite{FH19}, Freed and Hopkins computed the partition function of the Rarita-Schwinger anomaly theory, which is equal to
$$\hat\alpha_{RS}(W)=\exp\left(2\pi \sqrt{-1}\,\frac{\eta(TW-2)}{4}\right),$$
where $\eta(TW-2)$ is the difference of the $\eta$-invariant of the Rarita-Schwinger operator and twice the $\eta$-invariant of the pure Dirac operator. They showed that this partition function is a root of unity, independent of the metrics on $W$ and a pin$^+$ bordism invariant. This consequently shows that the 
Rarita-Schwinger partition function factors through a homomorphism
$$\hat\alpha_{RS}: \pi_{12} MTPin^+\to \CC^\times, $$
where $MTPin^+$ is the Thom spectra of pin$^+$-manifolds
and then determines an invertible unitary topological field theory
$$\alpha_{RS}: MTPin^+\to\Sigma^{12}I \CC^\times$$
with $I \CC^\times$ being the character dual to the sphere spectrum. On the other hand, to handle the $C$-field, Freed and Hopkins introduced a new topological structure, namely the $\mathfrak{m}_c$ structure, a pin$^+$-structure together with a $\omega_1$-twisted integer lift of the fourth Stiefel-Whitney class $\omega_4$. The anomaly of the $C$-field is
$$ \hat\alpha_C(W)=\exp\left(2\pi \sqrt{-1}\,\frac{\tilde c^3-\bar p(TW)\tilde c}{48}\right),$$
where $\bar p(TW)$ is a degree 4 canonical class of the pin$^+$ structure and $\tilde c$ is a $\omega_1$-twisted integer lift
of $\omega_4(TW)$. They showed that this factors through a homomorphism 
$$\hat\alpha_{C}: \pi_{12} M\mathfrak{m}_c\to \CC^\times,$$ 
where $M\mathfrak{m}_c$ is the Thom spectra of $\mathfrak{m}_c$-manifolds, and then determines an invertible topological field theory
$$\alpha_{C}: M\mathfrak{m}_c\to\Sigma^{12}I \CC^\times.$$
The following theorem shows that $M$-theory is anomaly free. 
\begin{theorem}[\protect Freed-Hopkins, \cite{FH19}] The total anomaly theory $\alpha_{RS}\otimes \alpha_{C}$ is trivializable. 
\end{theorem}
Freed and Hopkins proved this theorem by determining the generators for the 12-dimensional bordism group of $\mathfrak{m}_c$-manifolds after $2$-adic completion and verifying that $\hat\alpha_{RS}\cdot \hat\alpha_{C}=1$ on those generators. 

The anomaly cancellation under the assumption that $W$ is spin was proved by Witten in \cite{W97}. When $W$ is spin, the partition function of Rarita-Schwinger anomaly theory can be expressed via characteristic numbers of $W$ (Proposition 3.5 in \cite{FH19})
$$ \hat\alpha_{RS}(W)=\exp\left(2\pi \sqrt{-1}\left\langle\frac{\hat A(TW)\ch(T_\CC W-2)}{4}, [W]\right\rangle\right). $$
Let $\lambda=\lambda(W)$ be the first spin class of $TW$ in $H^4(W; \ZZ)$ and $x\in H^4(W; \ZZ)$. Let $c=C(x)=\lambda+2x$. The anomaly of the $C$-field is
$$ \hat\alpha_C(W)=\exp\left(2\pi \sqrt{-1}\,\left\langle\frac{ c^3-p(TW)c}{48}, [W]\right\rangle\right),$$
where $p(TW)$ be the second spin class of $TW$ in $H^8(W; \ZZ)$.
The anomaly cancellation $\hat\alpha_{RS}(W) \hat\alpha_C(W)=1$ is equivalent to the integrality of the following characteristic number,
$$ \left\langle\frac{ c^3-p(TW)c}{48}+\frac{1}{4}\hat A(TW)\ch(T_\CC W-2), [W]\right\rangle.$$ By the Atiyah-Hirzebruch divisibility, $\langle \hat A(TW), [W]\rangle$ is even. Therefore the anomaly cancellation is further equivalent to the integrailty of 
$$ \left\langle\frac{ c^3-p(TW)c}{48}+\frac{1}{4}\hat A(TW)\ch(T_\CC W-4), [W]\right\rangle.$$ Witten showed that $x$ determines an $E_8$-bundle $V_\CC(x)$ on $W$ and the above characteristic number is equal to minus half of the index of the Dirac operator $D_W$ on $W$ coupled with $V_\CC (x)$ by proving the following amazing equality through computation
\be \label{WFH1} \left\langle\frac{ c^3-p(TW)c}{48}+\frac{1}{2}\hat A(TW)\ch (V_\CC (x))+\frac{1}{4}\hat A(T	W)\ch(T_\CC W-4), [W]\right\rangle=0. \ee Then the desired integrality comes from the Atiyah-Hirzebruch divisibility on the eveness of $\ind (D_W^{V_\CC(x)})$.

Motivated by the $C$-field anomaly, Freed and Hopkins developed an algebraic theory of cubic forms in \cite{FH19}. We also briefly review their theory here. Let $L$ be a finitely generated free abelian group. Let 
\be \langle \cdot, \cdot, \cdot\rangle: L\times L \times L \to \ZZ \ee
be a symmetric trilinear form on $L$. For convenience, write the trilinear form simply as a product. $a \in L$ is called a {\em characteristic element} if the $\overline a\in L\otimes \ZZ/2\ZZ$ satisfies the following identity
\be \label{char} {\overline a} \, {\overline x} \, {\overline y} \equiv {\overline x}\, {\overline x}\, {\overline y}+{\overline x}\, {\overline y}\,{\overline y}\ \ \ (\!\!\!\!\!\!\mod\, 2)\ \  \ \ \  \ee 
for any $\overline x$, $\overline y\in L\otimes \ZZ/2\ZZ$.
Let $L_{\mathrm{char}}\subset L$ be the torsor of characteristic elements in $L$. Let $L^*=\mathrm{Hom} (L; \ZZ).$ 
\begin{theorem}[\protect Freed-Hopkins, Lemma 4.1 in \cite{FH19}] \label{FH-mod 24} Let $a\in L_{\mathrm{char}}$ and $\widehat a$ be the mod 24 reduction. There exists a unique $\widehat b\in L^*\otimes \ZZ/24\ZZ$ such that 
\be \label{mod24} \widehat b(\widehat x)=4\widehat x^3+6\widehat a\widehat x^2+3\widehat a^2\widehat x \ \ \ \ \ \ (\!\!\!\!\!\! \mod\, 24)\ee
for all $\widehat x\in L\otimes \ZZ/ 24 \ZZ.$
\end{theorem}

Let $h(x)$ be a degree 3 polynomial whose highest term is $\frac{1}{6}x^3$. Then the following algebraic identity shows that $h$ is a cubic refinement of the trilinear form $xyz$:
\be \label{cubicrefine} xyz=h(x+y+z)-h(x+y)-h(x+z)-h(y+z)+ h(x)+h(y)+h(z)-h(0). \ee

Let $a\in L, b\in L^*$. Consider the Witten-Freed-Hopkins polynomial on $L$:
\be f_{a, b}(x)=(a+x)^3-b(a+x).  \ee
It is easy to see that \be \label{mod48} \frac{f_{a, b}(2x)-f_{a, b}(0)}{48} \ee
is a degree 3 polynomial whose highest term is $\frac{1}{6}x^3$, and therefore a cubic refinement of $xyz$; moreover when $a$ is a characteristic element and $b$ satisfies (\ref{mod24}), 
\be\frac{f_{a, b}(2x)-f_{a, b}(0)}{48}\in \ZZ.\ee

$\, $

Back to the anomaly cancellation in the spin case, $L=H^4(W; \ZZ)$ with the symmetric trilinear form of the intersection pairings gives a geometric model for the Freed-Hopkins' algebraic theory of cubic forms. They showed that the spin class $\lambda(W)$ is a characteristic element and the formula (\ref{WFH1}) tells us that the spin characteristic classes $a=\lambda(W), b=p(W)$ solve the mod 24 equation (\ref{mod24}). Moreover 
from (\ref{cubicrefine}), we can see that the trilinear form of the cup product
$$\langle x\cup y\cup z, [W]\rangle$$ for $x, y, z\in H^4(W; \ZZ)$ has an integral cubic refinement. 

The first purpose of this paper is to show that the amazing equality (\ref{WFH1}) for spin manifolds can be obtained by the modularity of a modular form, called {\em twisted Witten class}, inspired by the theory of elliptic genus \cite{O87, LS88, W, Lan, HBJ, Liu1, Gri, Hop}. 

Moreover, this modular method can be generalized to spin$^c$ case and allows us to obtain a spin$^c$ version of (\ref{WFH1}) with a new spin$^c$ cubic form. Consequently, using index theorem for spin$^c$ Dirac operators, we can find spin$^c$ classes $a=\lambda_c, b=p_c$ on spin$^c$ manifolds, that solve the weakened (mod 12) congruence equation (\ref{mod24}) and see that on 12 dimensional spin$^c$ manifolds $W$, $$2\langle x\cup y\cup z, [W]\rangle$$ has an integral cubic refinement. On the other hand, if we stick to the original spin cubic form, then it loses analytic interpretations as indices of twisted Dirac operators on $W$. Nevertheless, we find that the analytic interpretations can be rescued on spin$^c$ and spin$^{\omega_2}$ manifolds by applying Zhang's Rokhlin congruence formulas via mod 2 indices on 10 dimensional spin or pin$^-$ manifolds.

If we further weaken the assumption from spin$^c$ to be general orientable manifolds, the modular method still works. Actually we are able to obtain an orientable version of (\ref{WFH1}) with an orientable cubic form. Consequently, using index theorem for twisted signature operators, we find characteristic classes $a, b$ on orientable manifolds, that solve the weakened (mod 3) congruence equation (\ref{mod24}) and see that on 12 dimensional orientable manifolds $W$, $$8\langle x\cup y\cup z, [W]\rangle$$ has an integral cubic refinement.

In all the cases, the Witten-Freed-Hopkins type formulas that we have obtained like (0.1) from the modular method are local, i.e., they hold on the level of differential forms. This allows us to use the Atiyah-Patodi-Singer index theorem to generalize them to manifolds with boundaries. 

In the following, let us give a more detailed account of the background and our work in various cases. 

\subsection{Spin case} \label{intro-spin}
Let $Z$ be a 12 dimensional smooth manifold. Denote the integral Pontrjagin classes and the Stiefel-Whitney classes of $Z$ by $p_i, \omega_i$ respectively. Let $x\in H^4(Z; \ZZ)$. Following Witten \cite{W97}, $x$ determines an isomorphism class of principal $E_8$ bundles on $Z$. Let $V(x)$ denote the real adjoint vector bundle associated to the principal $E_8$ bundle determined by the class $x$. Denote by $V_\CC(x)$ the complexification of $V(x)$. The Chern character of $V_\CC(x)$ is (c.f. (4.25) in \cite{FH19})
\be \ch (V_\CC(x))=248-60x+6x^2-\frac{1}{3}x^3.\ee
So by the expression of the Chern character in terms of the Chern classes, it is easy to see that 
\be x=\frac{1}{60}c_2(V_\CC (x)). \ee

Suppose $Z$ is closed and oriented. Let $L=H^4(Z; \ZZ)$ and the trilinear form is the cup product of three elements in $L$ evaluated on the fundamental class $[Z]$. 

Further suppose $Z$ is spin. There is a canonical degree 4 class $\lambda\in H^4(Z; \ZZ)$ such that $2\lambda=p_1$ and $\lambda\equiv \omega_4\, (\!\!\!\!\mod2)$ (\cite{FH19}). The following theorem shows that $\lambda$ is a characteristic element.
\begin{theorem} [\protect Freed-Hopkins, Lemma 4.4 in \cite{FH19}] The Stiefel-Whitney class $\overline \lambda=\omega_4$ of a closed spin 12-manifold satisfies (\ref{char}) and thus $\lambda$ is a characteristic element. 
\end{theorem}

For $a=\lambda$, Freed and Hopkins showed that there is a characteristic class $p\in H^8(Z; \ZZ)$ such that $2p=p_2-\lambda^2$ and when $b=p$, (\ref{mod24}) is satisfied. This is deduced from the beautiful and important anomaly cancellation formula discovered by Witten-Freed-Hopkins (Theorem \ref{WFH-main}), which is proved by a direct computation, as well as the famous Atiyah-Hirzebruch divisibility on $8k+4$ dimensional spin manifolds. Adopting the above notations, one sees that on a closed spin smooth 12-manifold $Z$, 
\be \frac{f_{\lambda, p}(2x)}{48} \ee
is a half integer, which has analytic meaning as (quarter of) indices of certain twisted Dirac operators on $Z$ (Theorem \ref{WFH-main}). 

On the other hand, let $M$ be a $4m$ dimensional compact oriented smooth manifold. Let $$\{\pm
2\pi \sqrt{-1}z_{j},1\leq j\leq 2m\}$$ denote the formal Chern roots of $T_{%
\mathbb{C}}M $, the complexification of the tangent vector bundle $TM$ of $M$%
. The famous Witten genus of $M$ can be written as 
\begin{equation*}
W(M)=\left\langle \left( \prod_{j=1}^{2m}z_{j}\frac{\theta ^{\prime }(0,\tau
)}{\theta (z_{j},\tau )}\right) ,[M]\right\rangle \in \mathbb{Q}[[q]],
\end{equation*}%
with $\tau \in \mathbb{H}$, the upper half-plane, and $q=e^{2\pi \sqrt{-1}%
\tau }$. The Witten genus was first introduced in \cite{W} and can be viewed as the
loop space analogue of the $\widehat{A} $-genus. It can be expressed as a $q$%
-deformed $\widehat{A}$-genus as
\begin{equation*}
W(M)=\left\langle \widehat{A}(TM)\mathrm{ch}\left( \Theta \left(T_{\mathbb{C%
}}M\right) \right) ,[M]\right\rangle ,
\end{equation*}%
where
\begin{equation*}
\Theta (T_{\mathbb{C}}M)=\overset{\infty }{\underset{n=1}{\otimes }}%
S_{q^{n}}(\widetilde{T_{\mathbb{C}}M}),\ \ {\rm with}\ \
\widetilde{T_{\mathbb{C}}M}=T_{\mathbb{C}}M-{\mathbb C}^{4m},
\end{equation*}%
is the Witten bundle introduced in \cite{W}. When the manifold $M$ is spin, according to the
Atiyah-Singer index theorem, the Witten genus can be expressed analytically as the index of the twisted Dirac operator, $$W(M)=\ind (D\otimes \Theta \left( T_{\mathbb{C%
}}M\right))\in \mathbb{Z}[[q]],$$ where $D$ is the Atiyah-Singer spin Dirac operator on $M$. Moreover, if $M$ is string, i.e. $$\lambda=\frac{1}{2}p_{1}(TM)=0,$$ or even weaker, if $M$ is spin and the
first rational Pontrjagin class of $M$ vanishes, then $W(M)$ is a
modular form of weight $2m$ over $SL(2,\mathbb{Z})$ with integral
Fourier development (\cite{Za}). The homotopy theoretical
refinement of the Witten genus on string manifolds leads to the beautiful
theory of tmf ({\em topological modular form}) developed by Hopkins and
Miller \cite{Hop}. The string condition is the orientablity condition for this generalized cohomology theory.

If the string condition $\lambda=0$ does not hold, one constructs the cohomology class (c.f. \cite{Gri}), 
\be \mathcal{W}(TM)=e^{\frac{1}{24}E_2(\tau)\cdot p_1(TM)}\widehat{A}(TM)\mathrm{ch}\left( \Theta \!\left( T_{\mathbb{C%
}}M\right) \right)\in H^{4*}(M; \QQ)[[q]],\ee where $E_2(\tau)$ is the Eisenstein series of weight 2 (c.f. Chap 2.3 in \cite{BGHZ}). We call $\mathcal{W}(TM)$ the {\em Witten class} of $M$.

Let $P$ be a principal $E_8$ bundle over $M$. In Section \ref{proof of spinc-main}, we consider an associated element $\mathcal{V}\in K(M)[[q]]$ constructed from the basic representation of affine $E_8$. 
Let $P_1, P_2$ be two principal $E_8$ bundles with the corresponding $\mathcal{V}_1, \mathcal{V}_2$. 
Let $W_1, W_2$ be the complexified vector bundles associated to the adjoint representation of $E_8$. 
Denote
$$\varphi(\tau)=\prod_{n=1}^\infty (1-q^n).$$
We construct the {\em twisted Witten class} (specified to the case when $\xi$ is trivial and $c=0$ in (\ref{QP1P2}) for the spin case here for simplicity)
\be \mathcal{Q}(\mathcal{V}_1, \mathcal{V}_2):=e^{\frac{1}{24}E_2(\tau)\cdot\left(\frac{1}{30}(c_2(W_1)+c_2(W_2)\right)}\mathcal{W}(TM)\varphi(\tau)^{16}\mathrm{ch}(\mathcal{V}_1)\mathrm{ch}(\mathcal{V}_2)\in H^{4*}(M; \QQ)[[q]],\ee
and show that the degree 12 component is a modular form of weight 14 over $SL(2,\mathbb{Z})$ when $M=Z$ is 12 dimensional. Using the fact that the space of modular forms of weight 14 over $SL(2,\mathbb{Z})$ is 1-dimensional and spanned by $E_4^2(\tau)E_6(\tau)$, where $E_4(\tau)$ and $E_6(\tau)$ are the Eisenstein series of weight 4 and 6 respectively (c.f. Chap 2.1 in \cite{BGHZ}), we deduce a factorization formula (\ref{spinc-form1}). This formula reduces to the Witten-Freed-Hopkins anomaly cancellation formula (\ref{WFH-mainformula}) (when $\xi$ is trivial and $c=0$). 

Back to the algebraic theory of cubic forms, consider the polynomial
\be \widetilde f_{a, b}(x)=4(a+x)^3-6a(a+x)^2-(b-3a^2)(a+x) \ee
with understanding $3a^2$ as an element in $L^*$ by abusing notations. 

It is easy to check that 
\be \label{old-new} \widetilde f_{a, b}(x)=\frac{f_{a, b}(2x)+f_{a, b}(0)}{2}. \ee
From (\ref{cubicrefine}), (\ref{mod48}) and (\ref{old-new}), we see that $ \frac{\widetilde f_{a, b}(x)-\widetilde f_{a, b}(0)}{24}$
is a cubic refinement of $xyz$, and when $a$ is a characteristic element and $b$ modulo 24 satisfies (\ref{mod24}), the following holds,
\be \frac{\widetilde f_{a, b}(x)-\widetilde f_{a, b}(0)}{24}\in \ZZ. \ee 

In view of (\ref{old-new}) and the half integrality of $\frac{f_{\lambda, p}(2x)}{48}$ as well as its analytic meaning, we see that on a closed 12 dimensional spin manifold $Z$, when $a=\lambda, b=p$, 
\be \frac{\widetilde f_{\lambda, p}(x)}{24}\ee is an integer, which
has analytic meaning as (half of) the indices of certain twisted Dirac operator on $Z$ (Theorem \ref{spin-new}). 

In Section \ref{proof of spinc-new}, we show that Theorem \ref{spin-new} can actually also be deduced from a factorization formula (\ref{spinc-form2}) proved there by constructing the {\em twisted Witten class} (specified to the case when $\xi$ is trivial and $c=0$ in (\ref{RPi}) for the spin case here for simplicity)
\be \mathcal{R}(\mathcal{V})=e^{\frac{1}{24}E_2(\tau)\cdot\frac{1}{30}c_2(W)}\mathcal{W}(TZ)\varphi(\tau)^{8}\mathrm{ch}(\mathcal{V})\in H^{4*}(Z; \QQ)[[q]]. \ee

To consider $\widetilde f_{\lambda, p}(x)$ defined via (\ref{old-new}) might look redundant. 
However we include it here because first it arises from the modular form $\mathcal{R}(\mathcal{V})$, different from $\mathcal{Q}(\mathcal{V}_1, \mathcal{V}_2)$; secondly, we would like to point out that the relation (\ref{old-new}) between the cubic form $\frac{\widetilde f_{\lambda, p}(x)}{24}$ and the cubic form $\frac{f_{\lambda, p}(2x)}{48}$ corresponds exactly to the relation between the corresponding modular forms:
\be \mathcal{R}(\mathcal{V})=\sqrt{\mathcal{Q}(\mathcal{V}, \mathcal{V})\cdot \mathcal{W}(TZ)}\in H^{4*}(Z; \QQ)[[q]]. \ee

\begin{remark} The method of constructing $\mathcal{Q}(\mathcal{V}_1, \mathcal{V}_2)$, $\mathcal{R}(\mathcal{V})$ and using their modularities to prove factorization formulas appeared in \cite{HLZ2} for fiber bundles with 10 dimensional fibers. In this paper, we apply this method to 12 dimensional manifolds. 
\end{remark}

In \cite{FH19}, Freed and Hopkins showed that for a pin$^+$ 12-manifold $Z$ with $\mathfrak{m}_c$ structure there exists a characteristic class $\widetilde c\in H^4(Z; \widetilde \ZZ)/\mathrm{torsion}$, which is a characteristic element in $$L=H^4(Z; \widetilde \ZZ)/\mathrm{torsion}$$ and a characteristic class $\overline p\in H^8(Z; \ZZ)/\mathrm{torsion}$ such that (\ref{mod24}) is satisfied. Then 
\be \frac{f_{\widetilde c, \overline p}(0)}{48} \ee
is a half integer. Freed-Hopkins (Theorem 2.2, Theorem 6.2 in \cite{FH19})
proved the following anomaly cancellation\be \exp{\left( -2\pi i\cdot \frac{f_{\widetilde c, \overline p}(0)}{48}\right)}\cdot \exp{ \left(2\pi i\cdot \frac{\eta(TZ-2)}{4}\right)}=1. \ee
We have not recovered this result via modularity yet.

\subsection{Spin$^c$ and spin$^{\omega_2}$ cases} \label{intro-spincspinw2} Suppose $Z$ is a closed 12 dimensional smooth manifold not necessarily spin. Then the classes $\lambda$ and $p$ in the previous subsection for the spin case do not necessarily exist. By abusing notations we just denote $\frac{1}{2}p_1$ by $\lambda$, and $\frac{1}{2}\left(p_2-\frac{1}{4}p_1^2\right)$ by $p$ in $H^*(Z; \QQ)$. The original cubic forms $ \frac{f_{\lambda, p}(2x)}{48}$ and $\frac{\widetilde f_{\lambda, p}(x)}{24} $
now only take values in rationals rather than integers and lose analytic interpretations as indices of twisted Dirac operators.

In Section \ref{originalcubic}, we will show that the analytic interpretations can be rescued on spin$^c$ and spin$^{\omega_2}$ manifolds by applying the Rokhlin congruence formulas established in \cite{Z93, Z94, Z09, Z17} via mod 2 indices on 10 dimensional spin or pin$^-$ manifolds. Let us be more precise in the following. 

Let $K$ be an $8k+4$ dimensional spin$^c$ manifold. Let $\xi$ be the complex line bundle of the spin$^c$ structure. Let $c=c_1(\xi)\in H^2(K; \ZZ).$ Let $U$ be a {\em characteristic submanifold} of the spin$^c$ structure, i.e. an orientable $8k+2$ dimensional submanifold of $K$ such that $[U]\in H_{8k+2}(K; \ZZ)$ is dual to $c$. $U$ carries a canonically induced spin structure up to spin cobordism. Let $D_U$ be the Atiyah-Singer spin Dirac operator on $U$. 

The Rokhlin congruence formula (\ref{Z2}) in Theorem \ref{Rok} allows one to write the twisted $\widehat A$-genus on $K$ in terms of \ \ $\!\!\!\!\!\!\mod2$ indices of twisted Dirac operators on $U$ with a correction term. Combining (\ref{Z2}) with the Witten-Freed-Hopkins anomaly cancellation formula (\ref{WFH-mainformula}) and the new formula (\ref{spin-newformula}), we obtain Theorem \ref{spinc-mod2theorem}. 

We call an $8k+4$ dimensional closed smooth oriented manifold $K$ a {\em spin$^{\omega_2}$ manifold} if there exists a rank 2 nonorientable real vector bundle $\xi$ such that $\omega_2(TK)=\omega_2(\xi)$. For such manifolds, the corresponding Rokhlin congruence formula has been studied in \cite{Z94}. Let $U$ be a {\em characteristic submanifold} of the spin$^{\omega_2}$ structure, i.e. a nonorientable $8k+2$ dimensional submanifold of $K$ such that $[U]\in H_{8k+2}(K; \ZZ/2\ZZ)$ is dual to $\omega_2(TK)\in H^2(K; \ZZ/2\ZZ)$. $U$ carries a canonically induced pin$^-$ structure up to pin$^-$ cobordism. 

The Rokhlin congruence formula (\ref{Zpin-}) in Theorem \ref{Rok-pin-} allows one to write the twisted $\widehat A$-genus of $K$ in terms of twisted \ $\!\!\!\!\!\!\mod2$ analytic indices on $U$ with a correction term. Combining (\ref{Zpin-}) with the Witten-Freed-Hopkins anomaly cancellation formula (\ref{WFH-mainformula}) and the new formula (\ref{spin-newformula}), we obtain Theorem \ref{spinw2-mod2theorem}. Details about spin$^{\omega_2}$ structures and the obstruction classes to them will be studied in Section \ref{sectionspinw2}.

Section \ref{spinc-newcubic} provides another way on spin$^c$ manifolds to restore the beautiful nature of the Witten-Freed-Hopkins anomaly cancellation formula on spin manifolds. Generally there are no characteristic elements on spin$^c$ manifolds like $\lambda$ for spin manifolds. In the algebraic theory of cubic forms, suppose $a\in L$ is not necessarily a characteristic element, then the Freed-Hopkins Theorem \ref{FH-mod 24} is weakened to mod 12. More precisely, if $\widehat a$ is the mod 12 reduction, then there exists a unique $\widehat b\in L^\ast\otimes \ZZ/12\ZZ$ such that 
\be \label{mod12} \widehat b(\widehat x)=4\widehat x^3+6\widehat a\widehat x^2+3\widehat a^2\widehat x\ \ \ \ \ (\!\!\!\!\!\!\mod \, 12)\ee
for all $\widehat x\in L\otimes \ZZ/ 12 \ZZ.$ 

Let $Z$ be a 12 dimensional closed smooth spin$^c$ manifold. Let $c$ be the first Chern class of the complex line bundle of the spin$^c$ structure. In Section \ref{spinc-newcubic}, we will show that on $Z$ there exist characteristic classes $a=\lambda_c \in H^4(Z; \ZZ), b=p_c\in H^8(Z; \ZZ)$ such that (\ref{mod12}) holds. This is derived from Theorem \ref{spinc-main} and Theorem \ref{spinc-new}, in which we show that 
\be \frac{f_{\lambda_c, p_c}(2x)}{24}\ee
is a half integer and 
\be \frac{\widetilde f_{\lambda_c, p_c}(x)}{12} \ee is an integer, by demonstrating their analytic meanings using indices of twisted spin$^c$ Dirac operators. 

The left hand sides of Theorem \ref{spinc-main} and Theorem \ref{spinc-new} provide new cubic forms on spin$^c$ manifolds, generalizing the cubic forms in the spin case when the manifold is spin and $c=0$. The coefficients appearing in the new cubic forms will be studied in Section \ref{coeff}.

Clearly $$h(x)=\frac{f_{\lambda_c, p_c}(2x)-f_{\lambda_c, p_c}(0)}{24}$$ is a polynomial of $x$ valued in $\ZZ$ with highest term $\frac{1}{3}x^3$. We therefore see that on spin$^c$ 12-manifolds, there is an integral cubic refinement only for $2xyz$ rather than $xyz$,
\be 2\langle x\cup y\cup z, [Z]\rangle=h(x+y+z)-h(x+y)-h(x+z)-h(y+z)+ h(x)+h(y)+h(z)-h(0). \ee

Theorem \ref{spinc-main} and Theorem \ref{spinc-new} are deduced from the factorization formulas (\ref{spinc-form1}) and (\ref{spinc-form2}), which are proved in Section \ref{proof-spinc} by constructing the generalized Witten class 
\be \mathcal{W}_c(TZ):=e^{\frac{1}{24}E_2(\tau)\cdot (p_1(TZ)-3c^2)}\widehat{A}(TZ)\mathrm{ch}\left( \Theta \left( T_{\mathbb{C%
}}Z, \xi_\CC\right) \right)\in H^{4*}(Z; \QQ)[[q]]\ee
and the {\em twisted generalized Witten classes}
\be \mathcal{Q}_c(\mathcal{V}_1, \mathcal{V}_2):=e^{\frac{1}{24}E_2(\tau)\cdot\left(\frac{1}{30}(c_2(W_1)+c_2(W_2)\right)}\mathcal{W}_c(TZ)\varphi(\tau)^{16}\mathrm{ch}(\mathcal{V}_1)\mathrm{ch}(\mathcal{V}_2)\in H^{4*}(Z; \QQ)[[q]],\ee
\be \mathcal{R}_c(\mathcal{V}):=e^{\frac{1}{24}E_2(\tau)\cdot\frac{1}{30}c_2(W)}\mathcal{W}_c(TZ)\varphi(\tau)^{8}\mathrm{ch}(\mathcal{V})\in H^{4*}(Z; \QQ)[[q]].\ee

Applying the Rokhlin congruence (\ref{Z1}) in Theorem \ref{Rok} to Theorems \ref{spinc-main} and \ref{spinc-new}, we give mod 2 index interpretations of $\frac{f_{\lambda_c, p_c}(2x)}{12}$ and $\frac{\widetilde f_{\lambda_c, p_c}(x)}{12}$ in Theorem \ref{B}. Now on spin$^c$ manifolds, we have two types of mod 2 formulas: Theorem \ref{spinc-mod2theorem} and Theorem \ref{B}. 
Subtracting the corresponding sides of these formulas, we obtain Corollary \ref{differ}, which involves interesting quadratic forms on 10 dimensional spin manifolds and mod 2 indices. This motivates us to introduce an intersection pairing on 10 dimensional closed spin manifold in the presence of a complex line bundle. Using the computation of Stong on $\Omega_{11}^{spin}(K(\mathbb{Z}, 4))$ (\cite{Stong86}), we are able to obtain Theorem \ref{spinbcthm}, which is more general than Corollary \ref{differ}. The quadratic forms appearing in Theorem \ref{spinbcthm} are related to mod 2 indices and give interesting quadratic refinements of the intersection pairings. See Remark \ref{remark-spinbcthm}.

\subsection{Orientable case}
Let $Z$ be a 12 dimensional closed smooth oriented manifold without assuming any additional topological constraints. In this general situation, we are not able to find characteristic classes $a$ of degree 4 and $b$ of degree 8 such that the mod 12 equality (\ref{mod12}) holds for all $x\in H^4(Z; \ZZ)$. 

However, by our modularity method, we find that if 
$$b=4p_1^2-7p_2$$ and $\widehat b=b\, (\!\!\!\!\mod3)$, then the following mod 3 equality holds,
\be \label{mod3}  \widehat b(\widehat x)=4\widehat x^3\ \ \ \ \ (\!\!\!\!\!\!\mod\, 3)\ee
for all $\widehat x\in H^4(Z; \ZZ)\otimes \ZZ/ 3 \ZZ.$
In fact, in Theorem \ref{o1} and Theorem \ref{o2}, we will show that when $a=-p_1$, 
\be\frac{f_{a, b}(2x)}{6} \ee
is a half integer and 
\be \frac{\widetilde f_{a, b}(x)}{3}\ee is an integer, by demonstrating their analytic meanings as indices of twisted signature operators. 

Now let $$h(x)=\frac{f_{a, b}(2x)-f_{a, b}(0)}{6},$$ which is a polynomial of $x$ with value in $\ZZ$ and highest term $\frac{8}{6}x^3$. We therefore see that on general oriented 12-manifolds, there is an integral cubic refinement only for $8xyz$ rather than for $2xyz$ or $xyz$,
\be 8\langle x\cup y\cup z, [Z]\rangle=h(x+y+z)-h(x+y)-h(x+z)-h(y+z)+ h(x)+h(y)+h(z)-h(0). \ee

We prove Theorem \ref{o1} and Theorem \ref{o2} in Section \ref{proof-o1o2} by constructing the $\widehat{L}$-Witten class and twisted $\widehat{L}$-Witten classes (see (\ref{L-Witten}), (\ref{L-Witten-twofold}) and (\ref{L-Witten-single})).

\begin{remark} In view of Theorem 1 in \cite{Liu1}, we may obtain formulas more general than the ones presented in Section \ref{sectionorien}. In fact, let $F$ be a spin vector bundle of even rank over $M$ such that $$\frac{1}{2}p_1(F)=\frac{1}{2}p_1(M)$$ and $S^\pm(F)$ be the spinor bundles of $F$. In the construction of the $\widehat{L}$-Witten class (\ref{L-Witten}), one can replace the bundle 
$$\Theta_1 \left( T_{\mathbb{C}}M\right)\otimes \Theta_2 \left( T_{\mathbb{C}}M\right)\otimes\Theta_3 \left( T_{\mathbb{C}}M\right) $$
by
$$\Theta_1 \left( F_{\mathbb{C}}\right)\otimes \Theta_2 \left( F_{\mathbb{C}}\right)\otimes\Theta_3 \left( F_{\mathbb{C}}\right)$$
(see the construction of $\Theta_1, \Theta_2, \Theta_3$ in (\ref{Theta1}), (\ref{Theta2}, (\ref{Theta3})). Then the similar modularity method will deduce formulas giving analytic interpretations to some interesting new cubic forms (depending on the rank of $F$) via the indices of twisted Dirac operators $$D\otimes ((S^+(F)\oplus S^-(F))\otimes V_\CC(x)),$$
$$D\otimes ((S^+(F)\oplus S^-(F))\otimes T_\CC Z),$$ $$D\otimes((S^+(F)\oplus S^-(F))\otimes (\wedge^2(F_\CC)-S^2 (F_\CC))),$$ and
$$ D\otimes (S^+(F)\oplus S^-(F)).$$

\end{remark}

\subsection{Organization of the paper} In Section \ref{sectionspin}, we review the Witten-Freed-Hopkins formula (Theorem \ref{WFH-main}) and present the new type of anomaly cancellation formula (Theorem \ref{spin-new}). We point out that they are special cases of the corresponding anomaly cancellation formulas for spin$^c$ manifolds, which are given in Section \ref{spinc-newcubic}. As these formulas are consequences of the factorization formulas (\ref{spinc-form1}) and (\ref{spinc-form2}), which hold on the level of differential forms, Theorems \ref{WFH-main} and \ref{spin-new} have analogues for manifolds with boundary. They are stated in Theorem \ref{spin-boundary}. 

In Section \ref{sectionspinc}, we give the Witten-Freed-Hopkins anomaly cancellation formulas on 12 dimensional spin$^c$ and spin$^{\omega_2}$ manifolds. First in Section \ref{originalcubic}, we consider the original cubic forms as in the spin case. We use the Rokhlin congruence formulas in \cite{Z93, Z94, Z09, Z17} to express the original cubic forms as mod 2 indices on 10 dimensional characteristic submanifolds with correction terms. Then in Section \ref{spinc-newcubic}, for the spin$^c$ case, we give new cubic forms and anomaly cancellation formulas, which generalize the anomaly cancellation formulas in the spin case. We will also give the mod 2 congruence formulas for the new cubic forms as well as the formulas for manifolds with boundary. 

In Section \ref{sectionorien}, we present the anomaly cancellation formulas for 12 dimensional orientable manifolds. 

In Section \ref{proof}, the proofs of the main theorems (Theorems \ref{spinc-main}, \ref{spinc-new}, \ref{o1} and \ref{o2}) in Section \ref{sectionspinc} and Section \ref{sectionorien} will be given. Actually what we prove are the factorization formulas (\ref{spinc-form1}), (\ref{spinc-form2}), (\ref{orient-form1}) and (\ref{orient-form2}), which are all equalities on the levels of differential forms. 

In Section \ref{coeff}, we study the characteristic class coefficients appearing in the cubic forms in Theorems \ref{spinc-main}, \ref{spinc-new}, \ref{o1} and \ref{o2}. 

In Section \ref{sectionspinw2}, details about spin$^{\omega_2}$ structures and the obstruction classes to them are studied.

$$ $$

\noindent{\bf Acknowledgements.}
Fei Han was partially supported by the grant AcRF R-146-000-263-114 from National University of Singapore. 

Ruizhi Huang was supported by Postdoctoral International Exchange Program for Incoming Postdoctoral Students under Chinese Postdoctoral Council and Chinese Postdoctoral Science Foundation, and ``Chen Jingrun'' Future Star Program of AMSS.
He was also supported in part by Chinese Postdoctoral Science Foundation (Grant nos. 2018M631605 and 2019T120145), and National Natural Science Foundation of China (Grant no. 11801544 and 11688101). 

Kefeng Liu is partially supported by an NSF grant. 

Weiping Zhang was partially supported by the NSFC Grant no. 11931007, and Nankai Zhide Foundation.

\section{Cubic forms on spin 12-manifolds} \label{sectionspin}
 In this section, we review the Witten-Freed-Hopkins anomaly cancellation formula, present the new type of cancellation formula and point out that they can be deduced from the more general formulas for the spin$^c$ case in Section \ref{spinc-newcubic}. We will also present the corresponding formulas when the manifolds have boundary. 
 
Let $Z$ be a closed spin smooth 12-manifold. Recall that $\lambda\in H^4(Z; \ZZ)$ satisfies $2\lambda=p_1$.
\begin{theorem}[\protect Freed-Hopkins \cite{FH19}] \label{FH-p}There is a canonical degree 8 integral class $p\equiv \omega_8 \ (\!\!\!\!\mod\, 2)$ such that 
$$2p=p_2-\lambda^2.$$
\end{theorem} 

Let $V(x)$ denote the real adjoint vector bundle associated to the principal $E_8$ bundle determined by a class $x\in H^4(Z; \ZZ)$. Denote by $V_\CC(x)$ the complexification of $V(x)$. Let $$C(x)=\lambda+2x\in H^4(Z; \ZZ).$$ 

One has the following important formula,
\begin{theorem}[\protect Witten-Freed-Hopkins \cite{W97, FH19}] \label{WFH-main} The following identity holds,
\be \label{WFH-mainformula}   \left\langle  \frac{C(x)[p-C(x)^2]}{48} , [Z]  \right\rangle=\left\langle   \frac{1}{2}\widehat A(TZ)\ch (V_\CC(x))+\frac{1}{4}\widehat A(TZ)\ch (T_\CC Z)-\widehat A(TZ), [Z] \right\rangle.\ee
\end{theorem} 

\begin{proof} Taking $\xi=\CC$ and $c=0$ in Theorem \ref{spinc-main}, one obtains (\ref{WFH-mainformula}). 
\end{proof}

We also have the following new cancellation formula. Let $$\widetilde C(x)=\lambda+x\in H^4(Z; \ZZ).$$
\begin{theorem}\label{spin-new} Let $$\widetilde p=p-3\lambda^2.$$ The following identity holds,
\be \label{spin-newformula} \left\langle  \frac{\widetilde C(x)[\widetilde p+6\lambda\widetilde C(x)-4\widetilde C(x)^2]}{24} , [Z]       \right\rangle=\left\langle   \frac{1}{2}\widehat A(TZ)\ch (V_\CC(x))+\frac{1}{2}\widehat A(TZ)\ch (T_\CC Z)+122\widehat A(TZ), [Z] \right\rangle.\ee

\end{theorem}

\begin{proof} Taking $\xi=\CC$ and $c=0$ in Theorem \ref{spinc-new}, one obtains (\ref{spin-newformula}). 
\end{proof}

Now suppose $Z$ has boundary and let $Y$ be the boundary of $Z$ with the induced spin structure. Let $g^{TZ}, g^{V(x)}$ be a Riemannian metric on $TZ$ and a Euclidean metric on $V(x)$ respectively. 
Let $\nabla^{TZ}$ be the Levi-Civita connection on $TZ$ and $\nabla^{V(x)}$ be a Euclidean connection on $V(x)$. $g^{TZ}, g^{V(x)}, \nabla^{TZ}, \nabla^{V(x)}$ induce the corresponding Hermitian metrics and connections on $T_\CC Z$ and $V_\CC(x)$, the complexifications. Assume all the involved metrics and connections are of product structures near $\partial Z=Y$. 

Let $D_Y$ be the Atiyah-Singer Dirac operator on $Y$. Let $\reta$ denote the reduced $\eta$-invariant in the sense of Atiyah-Patodi-Singer \cite{APS}. 

Denote by $p_i(\nabla^{TZ})$ the $i$-th Pontrjagin form of $(TZ, \nabla^{TZ})$ (c.f. \cite{Z}). Denote by $\lambda(\nabla^{TZ})$ the characteristic form $\frac{1}{2}p_1(\nabla^{TZ})$ and by $p(\nabla^{TZ})$ the characteristic form $\frac{1}{2}p_2(\nabla^{TZ})-\frac{1}{8}p_1(\nabla^{TZ})^2.$ 

In the following, when a connection appears in a bracket of a characteristic class, we always mean the corresponding characteristic form determined by this connection. 

Denote 
$$\widecheck x=\frac{1}{60}c_2(V_\CC(x), \nabla^{V_\CC(x)}).$$ 

Let $$C(\widecheck x)=\lambda(\nabla^{TZ})+2\widecheck x$$ and $$\widetilde C(\widecheck x)=\lambda(\nabla^{TZ})+\widecheck x.$$

As (\ref{WFH-mainformula}) and (\ref{spin-newformula}) hold on the level of forms, by the Atiyah-Patodi-Singer index theorem \cite{APS}, we have the following formulas. 
\begin{theorem} \label{spin-boundary}
\be \frac{1}{24}\int_Z C(\widecheck x)[p(\nabla^{TZ})-C(\widecheck x)^2]\equiv\reta(D_Y^{V_\CC(x)})+\frac{1}{2} \reta(D_Y^{T_\CC Z})-2\reta(D_Y) \ \ \ \!\!\!\!\!\!\mod\ZZ;
\ee
and
\be \frac{1}{24}\int_Z \widetilde C(\widecheck x)[\widetilde p(\nabla^{TZ})+6\lambda(\nabla^{TZ})\widetilde C(\widecheck x)-4\widetilde C(\widecheck x)^2]\equiv\frac{1}{2}\reta(D_Y^{V_\CC(x)})+\frac{1}{2} \reta(D_Y^{T_\CC Z})+122\reta(D_Y)\ \ \ \!\!\!\!\!\!\mod\ZZ.
\ee
\end{theorem}

\section{Cubic forms on spin$^c$ and spin$^{\omega_2}$ 12-manifolds} \label{sectionspinc}

In this section, we extend the Witten-Freed-Hopkins anomaly cancellation formulas to 12 dimensional spin$^c$ and spin$^{\omega_2}$ manifolds. 

\subsection{The original cubic forms} \label{originalcubic} Suppose $Z$ is a closed 12 dimensional smooth manifold not necessarily spin. Then the characteristic classes $\lambda$ and $p$ in the above section for the spin case do not necessarily exist. By abusing notations we simply denote $\frac{1}{2}p_1$ by $\lambda$, and $\frac{1}{2}\left(p_2-\frac{1}{4}p_1^2\right)$ by $p$ in $H^*(Z; \QQ)$. 

Let $x\in H^4(Z; \ZZ)$. Let $V(x)$ and $V_\CC(x)$ be the same meaning as introduced in the beginning of Section \ref{intro-spin}.

\subsubsection{Spin$^c$ case} Let $K$ be an $8k+4$ dimensional spin$^c$ manifold. Let $\xi$ be the complex line bundle of the spin$^c$ structure. Let $c=c_1(\xi)\in H^2(K; \ZZ).$ Let $U$ be a {\em characteristic submanifold} of the spin$^c$ structure, i.e. an orientable $8k+2$ dimensional submanifold of $K$ such that $[U]\in H_{8k+2}(K; \ZZ)$ is dual to $c$. $U$ carries a canonically induced spin structure up to spin cobordism. Let $D_U$ be the Atiyah-Singer spin Dirac operator on $U$. 

Denote $i:U\hookrightarrow K$ the embedding. Let $N$ be the normal bundle over $U$ in $K$ and $e\in H^2(U; \ZZ)$ the Euler class of $N$. Clearly $i^*TK\cong TU\oplus N$.

Let $E$ be a real vector bundle over $K$. Then $i^*E$ is a real vector bundle over the spin manifold $U$. Let $\ind_2(i^*E)$ be the mod 2 index in the sense of Atiyah–Singer \cite{ASV} associated to $i^*E$, which is a spin cobordism invariant. Let $E_\CC$ be the complexification of $E$. 

We have the following analytic Rokhlin congruence formula.
\begin{theorem}[\protect Zhang \cite{Z93, Z94, Z09}] \label{Rok}
\be  \label{Z1}
\left\langle  \widehat A(TK)\exp\left(\frac{c}{2}\right)\ch(E_\CC), [K]  \right\rangle\equiv\ind_2 (D_U^{i^*E})\ \ \ \!\!\!\!\!\!\mod2
\ee
and 
\be \label{Z2} \left\langle \widehat A(TK)\ch(E_\CC), [K]  \right\rangle\equiv\ind_2 (D_U^{i^*E})-\frac{1}{2}\left\langle \widehat A(TU)\ch(i^*E_\CC)\tanh\left(\frac{e}{4}\right), [U]\right\rangle \ \ \ \!\!\!\!\!\!\mod2.\ee

\end{theorem}

Combining (\ref{Z2}) with Theorems \ref{WFH-main} and \ref{spin-new}, we have
\begin{theorem}\label{spinc-mod2theorem} If $Z$ is a 12 dimensional closed smooth spin$^c$ manifold and $U$ is a characteristic submanifold, then the following identities hold,
\be \label{spinc-mod2-1} 
\begin{split}
&\left\langle  \frac{C(x)[p-C(x)^2]}{12} , [Z]  \right\rangle\\
=&\left\langle  2\widehat A(TZ)\ch (V_\CC(x))+\widehat A(TZ)\ch (T_\CC Z)-4\widehat A(TZ), [Z] \right\rangle\\
\equiv&\ind_2 (D_U^{TU})+\ind_2 (D_U^{N})-\frac{1}{2}\left\langle \widehat A(TU)\ch\left(2i^*V_\CC(x)+T_\CC U+N_\CC-4\right)\tanh\left(\frac{e}{4}\right), [U]\right\rangle \ \ \ \!\!\!\!\!\!\mod2;
\end{split}
\ee
and
\be \label{spinc-mod2-2}
\begin{split} 
&\left\langle  \frac{\widetilde C(x)[\widetilde p+6\lambda\widetilde C(x)-4\widetilde C(x)^2]}{12} , [Z]       \right\rangle\\
=&\left\langle \widehat A(TZ)\ch (V_\CC(x))+\widehat A(TZ)\ch (T_\CC Z)+244\widehat A(TZ), [Z] \right\rangle\\
\equiv&\ind_2 (D_U^{i^*V(x)})+\ind_2 (D_U^{TU})+\ind_2 (D_U^{N})\\
&-\!\frac{1}{2}\!\left\langle \widehat A(TU)\ch(i^*V_\CC(x)+T_\CC U+N_\CC+244)\tanh\left(\frac{e}{4}\right), [U]\right\rangle \ \ \ \!\!\!\!\!\!\mod2.
\end{split}
\ee
\end{theorem} 

$\, $

\subsubsection{Spin$^{\omega_2}$ case} We call an $8k+4$ dimensional closed smooth oriented manifold $K$ a {\em spin$^{\omega_2}$ manifold} if there exists a rank 2 nonorientable real vector bundle $\xi$ such that $\omega_2(TK)=\omega_2(\xi)$. Such manifolds and the corresponding Rokhlin congruence have been studied in \cite{Z94} and \cite{Z17}. Let $U$ be a {\em characteristic submanifold} of the spin$^{\omega_2}$ structure, i.e., a nonorientable $8k+2$ dimensional submanifold of $K$ such that $[U]\in H_{8k+2}(K; \ZZ/2\ZZ)$ is dual to $\omega_2(TK)\in H^2(K; \ZZ/2\ZZ)$. $U$ carries a canonically induced pin$^-$ structure up to pin$^-$ cobordism. 

Denote $i:U\hookrightarrow K$ the embedding. Let $N$ be the normal bundle over $U$ in $K$ and $e\in H^2(U, o(TU))$ the Euler class of $N$.  

Let $E$ be a real vector bundle over $K$. Then $i^*E$ is a real vector bundle over the pin$^-$ manifold $U$. Let $\ind_2^{a}(i^*E)$ be the mod 2 analytic index of the real vector bundle $i^*E$ over $U$, which is defined via $\eta$-invariants (c.f. \cite{Z17}). We have the following Rokhlin congruence formula. 
\begin{theorem}[\protect Zhang (Theorem A.2 in \cite{Z17})] \label{Rok-pin-}
\be \label{Zpin-}
\left\langle \widehat A(TK)\ch(E_\CC), [K]  \right\rangle\equiv\ind^a_2 (i^*E)-\frac{1}{2}\left\langle \widehat A(TU)\ch(i^*E_\CC)\tanh\left(\frac{e}{4}\right), [U]\right\rangle \ \ \ \!\!\!\!\!\!\mod2.\ee

\end{theorem}

Combining (\ref{Zpin-}) with Theorem \ref{WFH-main} and Theorem \ref{spin-new}, we have

\begin{theorem}\label{spinw2-mod2theorem} If $Z$ is a 12 dimensional closed smooth spin$^{\omega_2}$ manifold and $U$ is a characteristic submanifold, then the following identities hold,
\be \label{spinw2-mod2-1} 
\begin{split}
&\left\langle  \frac{C(x)[p-C(x)^2]}{12} , [Z]  \right\rangle\\
=&\left\langle  2\widehat A(TZ)\ch (V_\CC(x))+\widehat A(TZ)\ch (T_\CC Z)-4\widehat A(TZ), [Z] \right\rangle\\
\equiv&\ind_2^a (TU)+\ind_2^a (N)-\frac{1}{2}\left\langle \widehat A(TU)\ch\left(2i^*V_\CC(x)+T_\CC U+N_\CC-4\right)\tanh\left(\frac{e}{4}\right), [U]\right\rangle \ \ \ \!\!\!\!\!\!\mod2;
\end{split}
\ee
and
\be \label{spinw2-mod2-2}
\begin{split} 
&\left\langle  \frac{\widetilde C(x)[\widetilde p+6\lambda\widetilde C(x)-4\widetilde C(x)^2]}{12} , [Z]       \right\rangle\\
=&\left\langle \widehat A(TZ)\ch (V_\CC(x))+\widehat A(TZ)\ch (T_\CC Z)+244\widehat A(TZ), [Z] \right\rangle\\
\equiv&\ind_2^a(i^*{V(x)})\!+\!\ind_2^a (TU)\!+\!\ind_2^a (N)\!\\
&-\!\frac{1}{2}\!\left\langle \widehat A(TU)\ch(i^*V_\CC(x)+T_\CC U+N_\CC+244)\tanh\left(\frac{e}{4}\right), [U]\right\rangle \ \ \ \!\!\!\!\!\!\mod2.
\end{split}
\ee
\end{theorem}

\subsection{New cubic forms on spin$^c$ manifolds} \label{spinc-newcubic}
Let $Z$ be a 12 dimensional closed smooth spin$^c$ manifold. Let $\xi$ be the complex line bundle of the spin$^c$ structure. We use $\xi_\RR$ for the notation of $\xi$ when it is viewed as an oriented real plane bundle. Let $c=c_1(\xi)\in H^2(Z; \ZZ).$ Denote $\xi_\CC=\xi_\RR\otimes_\RR \CC$. Clearly $\xi_\CC=\xi\oplus \overline \xi$. 

Since $Z$ is spin$^c$, $TZ\oplus \xi_\RR$ is spin. By a result of McLaughlin (Lemma 2.2, \cite{M}), there is a canonical class $\rho_c\in H^4(Z; \ZZ)$ associated to the spin$^c$ structure such that 
$$2\rho_c=p_1(TZ\oplus \xi_\RR)\in H^4(Z; \ZZ).$$ 
However,
\be \begin{split} &p_1(TZ\oplus \xi_\RR)\\
=&-c_2((TZ\oplus \xi_\RR)\otimes_\RR \CC)\\
=&-c_2(T_\CC Z)-c_2(\xi\oplus \overline \xi)-c_1(T_\CC Z)c_1(\xi\oplus \overline \xi)\\
=&p_1(TZ)+c^2.\end{split} \ee
So $$2\rho_c=p_1(TZ)+c^2.$$

Let $$\lambda_c:=\rho_c-2c^2\in H^4(Z; \ZZ)$$ and $$C_c(x)=\lambda_c+2x\in H^4(Z; \ZZ).$$

\begin{theorem} \label{spinc-main}There is a degree 8 integral class $p_c$ such that 
\be \label{spinc-class}
\begin{split}
&p_c\equiv \omega_8 ~({\rm mod}~2),\\
&8p_c=4p_2-p_1^2-6p_1c^2+39c^4.
\end{split}
\ee
Moreover, the following identity holds, 
\be \label{spinc-mainformula}
\begin{split}
&\left\langle  \frac{C_c(x)[p_c-C_c(x)^2]}{24} , [Z]  \right\rangle\\
=&\left\langle   \widehat A(TZ)e^{c/2}\ch (V_\CC(x))+\frac{1}{2}\widehat A(TZ)e^{c/2}\ch (T_\CC Z)-\frac{1}{2}\widehat A(TZ)e^{c/2}\ch[\xi_\CC\otimes \xi_\CC-\xi_\CC+2], [Z] \right\rangle.
\end{split}
\ee
\end{theorem}

The existence of $p_c\in H^8(Z; \ZZ)$ is proved in Theorem \ref{pc}. Let $$\widetilde p_c=p_c-3\lambda_c^2.$$ In Theorem \ref{pc}, it is also shown that 
\be 
\begin{split}
&\widetilde p_c\equiv \omega_8+\omega_4^2+\omega_2^4 ~({\rm mod}~2),\\
&8\widetilde p_c=4p_2-7p_1^2+30p_1c^2-15c^4.
\end{split}
\ee

Let $$\widetilde C_c(x)=\lambda_c+x\in H^4(Z; \ZZ).$$
\begin{theorem} \label{spinc-new} The following identity holds, 
\be \label{spinc-newformula}
\begin{split}
&\left\langle  \frac{\widetilde C_c(x)[\widetilde p_c+6\lambda_c\widetilde C_c(x)-4\widetilde C_c(x)^2]}{12} , [Z] \right\rangle\\
=&\left\langle   \widehat A(TZ)e^{c/2}\ch (V_\CC(x))+\widehat A(TZ)e^{c/2}\ch (T_\CC Z)+\widehat A(TZ)e^{c/2}\ch[-\xi_\CC\otimes \xi_\CC+\xi_\CC+246], [Z] \right\rangle.
\end{split}
\ee

\end{theorem}

These two theorems are consequences of the factorization formulas for degree 12 characteristic forms: (\ref{spinc-form1}) and (\ref{spinc-form2}) (proved in Section \ref{proof-spinc}) and a direct computation of the degree 8 components in (\ref{spinc-form1}) and (\ref{spinc-form2}).

$\, $

Let $U$ be a characteristic submanifold of $Z$. Let $N$ be the normal bundle over $U$ in $Z$. Applying the Rokhlin congruence formula (\ref{Z1}), we have

\begin{theorem}\label{B} The following identities hold, 
\be \label{spinc-B-1}
\left\langle  \frac{C_c(x)[p_c -C_c(x)^2]}{12} , [Z]  \right\rangle\equiv\ind_2 (D_U^{TU})+\ind_2 (D_U^{N\otimes N})\ \ \ \!\!\!\!\!\!\mod2;
\ee
and
\be \label{spinc-B-2}
\left\langle  \frac{\widetilde C_c(x)[\widetilde p_c+6\lambda_c\widetilde C_c(x)-4\widetilde C_c(x)^2]}{12} , [Z] \right\rangle\equiv \ind_2 (D_U^{TU})+\ind_2 (D_U^{N\otimes N})+\ind_2(D_U^{i^*V(x)})\ \ \ \!\!\!\!\!\!\mod2.
\ee

\end{theorem}
\begin{proof} Clearly $i^*TZ\cong TU\oplus N$ and $i^*\xi_\RR \cong N$. 

Combining Theorem \ref{spinc-main} with (\ref{Z1}), we have, when $\!\!\!\!\mod\, 2$,
\be 
\begin{split}
&\left\langle  \frac{C_c(x)[p_c-C_c(x)^2]}{12} , [Z]  \right\rangle\\
=&\left\langle   2\widehat A(TZ)e^{c/2}\ch (V_\CC(x))+\widehat A(TZ)e^{c/2}\ch (T_\CC Z)-\widehat A(TZ)e^{c/2}[\xi_\CC\otimes \xi_\CC-\xi_\CC+2], [Z] \right\rangle\\
=& 2\ind_2(D_U^{i^*V(x)})+\ind_2 (D_U^{TU})+\ind_2 (D_U^{N})+\ind_2(D_U^{N\otimes N})+\ind_2 (D_U^{N})+2\ind_2 (D_U)\\
=&\ind_2 (D_U^{TU})+\ind_2 (D_U^{N\otimes N}).
\end{split}
\ee
Combining Theorem \ref{spinc-new} with (\ref{Z1}), we have, when $\!\!\!\!\mod\, 2$,
\be 
\begin{split}
&\left\langle  \frac{\widetilde C_c(x)[\widetilde p_c+6\lambda_c\widetilde C_c(x)-4\widetilde C_c(x)^2]}{12} , [Z] \right\rangle\\
=&\left\langle   \widehat A(TZ)e^{c/2}\ch (V_\CC(x))+\widehat A(TZ)e^{c/2}\ch (T_\CC Z)+\widehat A(TZ)e^{c/2}[-\xi_\CC\otimes \xi_\CC+\xi_\CC+246], [Z] \right\rangle\\
=&\ind_2(D_U^{i^*V(x)})+\ind_2 (D_U^{TU})+\ind_2 (D_U^{N})+\ind_2 (D_U^{N\otimes N})+\ind_2 (D_U^{N})+246\ind_2 (D_U)\\
=&\ind_2 (D_U^{TU})+\ind_2 (D_U^{N\otimes N})+\ind_2(D_U^{i^*V(x)}).
\end{split}
\ee
The desired formulas follow. \end{proof}

By subtracting the corresponding sides of (\ref{spinc-mod2-1}) from (\ref{spinc-B-1}), and (\ref{spinc-mod2-2}) from (\ref{spinc-B-2}) respectively, and then applying the Poincar\'e duality, we can show 
the following through direct computations.
\begin{corollary} \label{differ}
Let $e$ be the Euler class of $N$ over $U$.  
The following identities hold, 
\be \label{differ1}
\begin{split}
&\frac{1}{64} \left\langle e\cdot \left\{24i^*C(x)^2-(4p_1(TU)+10e^2)i^*C(x)+p_1(TU)^2-4p_2(TU)+6p_1(TU)e^2-21e^4 \right\}, [U]\right\rangle\\
\equiv& \ind_2 (D_U^{N})+\ind_2 (D_U^{N\otimes N})+\frac{1}{2}\left\langle \widehat A(TU)\ch\left(2i^*V_\CC(x)+T_\CC U+N_\CC-4\right)\tanh\left(\frac{e}{4}\right), [U]\right\rangle \ \ \ \!\!\!\!\!\!\mod2;
\end{split}
\ee
and 
\be \label{differ2}
\begin{split}
&\frac{1}{64} \langle e\cdot \{48i^*\widetilde C(x)^2-(28p_1(TU)+10e^2)i^*\widetilde C(x)+7p_1(TU)^2-4p_2(TU)+6p_1(TU)e^2-21e^4 \}, [U]\rangle\\
\equiv& \ind_2 (D_U^{N})+\ind_2 (D_U^{N\otimes N})+\!\frac{1}{2}\!\left\langle \widehat A(TU)\ch(i^*V_\CC(x)+T_\CC U+N_\CC+244)\tanh\left(\frac{e}{4}\right), [U]\right\rangle \ \ \ \!\!\!\!\!\!\mod2.
\end{split}
\ee
\end{corollary}
\begin{remark} \label{remark-differ} Note that on the spin manifold $U$, $$i^*C(x)=\frac{1}{2}p_1(TU)+2i^*x, \ \ \ i^*\widetilde C(x)=\frac{1}{2}p_1(TU)+i^*x.$$ 
The left hand sides of (\ref{differ1}) and (\ref{differ2}) give interesting integral quadratic forms on $H^4(U; \ZZ)$ after being multiplied by 64.
\end{remark}
 
Motivated by Corollary \ref{differ} and Remark \ref{remark-differ}, for any $10$-dimensional closed spin manifold $B$ with an auxiliary complex line bundle $\xi$, we consider the following pairing on the cohomology classes of degree $4$: 
\[
\mathcal{L}^c: H^{4}(B;\mathbb{Z})\otimes H^{4}(B;\mathbb{Z})\rightarrow \mathbb{Z}
\]
defined by 
\[
\mathcal{L}^c(x, y)= \langle x \cup y \cup c, [B] \rangle,
\]
where $c:=c_1(\xi)\in H^2(B;\mathbb{Z})$ is the first Chern class of $\xi$ and $[B]$ is the fundamental class of $B$. This naturally defines a symmetric bilinear form on the free part of $H^{4}(B;\mathbb{Z})$, which we denote by $\mathcal{L}^{c}$.

Furthermore, it was shown by Stong \cite{Stong86} that the $11$-th spin bordism group over the Eilenberg-MacLane space $K(\mathbb{Z},4)$ is trivial, i.e., $\Omega_{11}^{spin}(K(\mathbb{Z}, 4))=0$. Since by Bott-Samelson \cite{BS58}, $K(\mathbb{Z},4)$ is homotopy equivalent to $BE_8$ up to $15$ skeleton, it is also true that
\[
\Omega_{11}^{spin}(BE_8)=0.
\]
This implies that the circle bundle $\mathbb{S}(\xi)$ of $\xi$, which is spin, bounds a $12$-dimensional spin manifold $W$ such that any $E_8$-principal bundle over $\mathbb{S}(\xi)$ can be extended to $W$. Again by Bott-Samelson \cite{BS58} that $K(\mathbb{Z},4)$ is homotopy equivalent to $BE_8$ up to $15$ skeleton, the isomorphism classes of $E_8$-principal bundles over any manifold $M$ of dimension less than $15$ are in one-one correspondence with the $4$-classes in $H^4(M;\mathbb{Z})$. Let $x\in H^4(B;\mathbb{Z})$ correspond to a given $E_8$-principal bundle over $B$. Since the disk bundle $\mathbb{D}(\xi)$ of $\xi$ is homotopy equivalent to $B$, this $E_8$-principal bundle is extended to $\mathbb{D}(\xi)$, and then restricted to $\mathbb{S}(\xi)$ with a further extension to $W$. Now gluing $\mathbb{D}(\xi)$ with $W$ along $\mathbb{S}(\xi)$, we get a $12$-dimensional closed spin$^c$ manifold 
\[
Z= \mathbb{D}(\xi)\bigcup_{\mathbb{S}(\xi)} W,
\]
such that there exists a principal $E_8$-bundles restricted to the given one over $B$. It is clear that $B$ is a characteristic submanifold of $Z$. Let $V(x)$ denote the real adjoint bundle over $B$ associated to the principal $E_8$-bundle determined by the class $x$.
By Corollary \ref{differ} and Remark \ref{differ}, we have
\begin{theorem}\label{spinbcthm}
\be \label{spinbc1}
\begin{split}
&\frac{1}{64} \left\langle c\cdot \left\{24C(x)^2-(4p_1(TB)+10c^2)C(x)+p_1(TB)^2-4p_2(TB)+6p_1(TB)c^2-21c^4 \right\}, [B]\right\rangle\\
\equiv& \ind_2 (D_B^{\xi_{\mathbb{R}}})+\ind_2 (D_B^{\xi_{\mathbb{R}}\otimes \xi_{\mathbb{R}}})+\frac{1}{2}\left\langle \widehat A(TB)\ch\left(2 V_\CC(x)+T_\CC B+\xi_\CC-4\right)\tanh\left(\frac{c}{4}\right), [B]\right\rangle \ \ \ \!\!\!\!\!\!\mod2;
\end{split}
\ee
and 
\be \label{spinbc2}
\begin{split}
&\frac{1}{64} \langle c\cdot \{48\widetilde C(x)^2-(28p_1(TB)+10c^2)\widetilde C(x)+7p_1(TB)^2-4p_2(TB)+6p_1(TB)c^2-21c^4 \}, [B]\rangle\\
\equiv& \ind_2 (D_B^{\xi_{\mathbb{R}}})+\ind_2 (D_B^{\xi_{\mathbb{R}}\otimes \xi_{\mathbb{R}}})+\!\frac{1}{2}\!\left\langle \widehat A(TB)\ch(V_\CC(x)+T_\CC B+\xi_\CC+244)\tanh\left(\frac{c}{4}\right), [B]\right\rangle \ \ \ \!\!\!\!\!\!\mod2.
\end{split}
\ee
where $C(x)=\frac{1}{2}p_1(TB)+2x$, and $\widetilde C(x)=\frac{1}{2}p_1(TB)+x$.
\end{theorem}

\begin{remark}\label{remark-spinbcthm} It is not hard to check that 
$$\frac{1}{64} \left\langle c\cdot \left\{24C(x)^2-(4p_1(TB)+10c^2)C(x)+p_1(TB)^2-4p_2(TB)+6p_1(TB)c^2-21c^4 \right\}, [B]\right\rangle,$$ 
and
$$
\frac{1}{64} \langle c\cdot \{48\widetilde C(x)^2-(28p_1(TB)+10c^2)\widetilde C(x)+7p_1(TB)^2-4p_2(TB)+6p_1(TB)c^2-21c^4 \}, [B]\rangle$$
are quadratic refinements of $3\LL^c$ and $6\LL^c$ respectively. 
\end{remark}

$\, $

Now suppose $Z$ has boundary and let $Y$ be the boundary of $Z$ with the induced spin$^c$ structure. Assume all the involved metrics and connections are of product structures near $\partial Z=Y$. Let $D^c_Y$ be the Atiyah-Singer spin$^c$ Dirac operator on $Y$.

Let $$C_c(\widecheck x)=\lambda_c(\nabla^{TZ})+2\widecheck x$$ and 
$$\widetilde C_c(\widecheck x)=\lambda_c(\nabla^{TZ})+\widecheck x.$$

As (\ref{spinc-form1}) and (\ref{spinc-form2}) hold on the level of forms, by the Atiyah-Patodi-Singer index theorem \cite{APS}, we have the following formulas,
\begin{theorem} The following identities hold, 
\be \frac{1}{12}\int_Z C_c(\widecheck x)[p_c(\nabla^{TZ})-C_c(\widecheck x)^2]\equiv 2\reta(D_Y^{c, V_\CC(x)})+\reta(D_Y^{c, T_\CC Z})-\reta(D_Y^{c, \xi_\CC\otimes \xi_\CC-\xi_\CC+2}) \ \ \ \!\!\!\!\!\!\mod\ZZ;
\ee
and
\be 
\begin{split}
&\frac{1}{12}\int_Z \widetilde C_c(\widecheck x)[\widetilde p_c(\nabla^{TZ})+6\lambda_c(\nabla^{TZ})\widetilde C_c(\widecheck x)-4\widetilde C_c(\widecheck x)^2]\\
\equiv&\reta(D_Y^{c, V_\CC(x)})+\reta(D_Y^{c, T_\CC Z})+\reta(D_Y^{c, -\xi_\CC\otimes \xi_\CC+\xi_\CC+246})\ \ \ \!\!\!\!\!\!\mod\ZZ.
\end{split}
\ee
\end{theorem}

$\, $

\section{Cubic forms on orientable 12-manifolds} \label{sectionorien}

In this section, we give the Witten-Freed-Hopkins type anomaly cancellation formulas on 12 dimensional orientable manifolds, without assuming that the manifold is spin, spin$^c$ or spin$^{\omega_2}$. 

Let $Z$ be a 12 dimensional oriented smooth closed manifold. Let \be \widehat L(x)=\frac{x}{\tanh \frac{x}{2}} \ee
be the $\widehat L$-polynomial and $\widehat L(TZ)$ be the $\widehat L$-class of $TZ$ (c.f. \cite{Liu1}). Let $d_s$ be the signature operator on $Z$ and $W$ be a complex vector bundle over $Z$. Then by the Atiyah-Singer index theorem 
(\cite{AS})
$$ \ind(d_s\otimes W)=\langle \widehat L(TZ)\ch(W), [Z]\rangle.$$

For $x\in H^4(Z; \ZZ)$, let $$D(x)=-p_1+2x.$$
\begin{theorem}\label{o1} 
The following identity holds,
\be \label{o1-formula}
\begin{split}
&\left \langle  \frac{D(x)[4p_1^2-7p_2-D(x)^2]}{6} , [Z]       \right\rangle\\
=&\frac{1}{32}\left\langle 2\widehat L(TZ)\ch (V_\CC(x))+2\widehat L(TZ)\ch (T_\CC Z)+\widehat L(TZ)\ch(\wedge^2(T_\CC Z)-S^2(T_\CC Z))-4\widehat L(TZ), [Z] \right\rangle.
\end{split}
\ee

\end{theorem} 
$\, $

Let $$\widetilde D(x)=-p_1+x.$$ 

\begin{theorem}\label{o2} The following identity holds,
\be \label{o2-formula}
\begin{split}
&\left \langle  \frac{\widetilde D(x)[p_1^2-7p_2-6p_1\widetilde D(x)-4\widetilde D(x)^2]}{3} , [Z]       \right\rangle\\
=&\frac{1}{16}\left\langle \widehat L(TZ)\ch (V_\CC(x))+2\widehat L(TZ)\ch (T_\CC Z)+\widehat L(TZ)\ch(\wedge^2(T_\CC Z)-S^2 (T_\CC Z))+244\widehat L(TZ), [Z] \right\rangle.
\end{split}
\ee

\end{theorem} 

These two theorems are consequences of the factorization formulas for degree 12 characteristic forms: (\ref{orient-form1}) and (\ref{orient-form2}) (proved in Subsection \ref{proof-o1o2}) and direct computations of the degree 8 components in (\ref{orient-form1}) and (\ref{orient-form2}). 

$\, $

Now suppose $Z$ has boundary and let $Y$ be the boundary of $Z$ with the induced orientation. Assume all the involved metrics and connections are of product structures near $\partial Z=Y$. Let $B_Y$ be the signature operator on $Y$. 

Denote $$D(\widecheck x)=-p_1(\nabla^{TZ})+2\widecheck x$$ and $$\widetilde D(\widecheck x)=-p_1(\nabla^{TZ})+\widecheck x.$$

As (\ref{orient-form1}) and (\ref{orient-form2}) hold on the level of forms, by the Atiyah-Patodi-Singer index theorem \cite{APS}, we have

\begin{theorem} \label{o-boundary} The following identities hold,
\be \begin{split}
&\frac{16}{3}\int_Z D(\widecheck x)[4p_1(\nabla^{TZ})^2-7p_2(\nabla^{TZ})-D(\widecheck x)^2]\\
\equiv&2\reta(B_Y^{V_\CC(x)})+2\reta(B_Y^{T_\CC Z})+ \reta(B_Y^{\wedge^2(T_\CC Z)-S^2 (T_\CC Z)})-4\reta(B_Y) \ \ \ \!\!\!\!\!\!\mod\ZZ;
\end{split}
\ee
and 
\be \begin{split}
&\frac{16}{3}\int_Z \widetilde D(\widecheck x)[p_1(\nabla^{TZ})^2-7p_2(\nabla^{TZ})-6p_1(\nabla^{TZ}) \widetilde D(\widecheck x)-4 \widetilde D(\widecheck x)^2]\\
\equiv&\reta(B_Y^{V_\CC(x)})+2\reta(B_Y^{T_\CC Z})+ \reta(B_Y^{\wedge^2(T_\CC Z)-S^2 (T_\CC Z)})+244\reta(B_Y) \ \ \ \!\!\!\!\!\!\mod\ZZ.
\end{split}
\ee
\end{theorem}

$\, $
\section{Proofs} \label{proof}
The purpose of this section is to give proofs to Theorem \ref{spinc-main}, Theorem \ref{spinc-new}, Theorem \ref{o1} and Theorem \ref{o2}. In order to conduct the proofs, we first briefly review some basic materials on the representation theory of affine $E_8$ following \cite{K1} (see also \cite{K2}).

\subsection{The basic representation of affine $E_8$} \label{affineE8} Let $\mathfrak{g}$ be the Lie algebra of $E_8$. Let $\langle\ , \rangle$ be the Killing form on $\mathfrak{g}$.
Let $\widetilde{\mathfrak{g}}$ be the affine Lie algebra corresponding to $\mathfrak{g}$ defined by
$$\widetilde{\mathfrak{g}}=\CC[t, t^{-1}]\otimes \mathfrak{g}\oplus \CC c, $$ with the bracket
$$[P(t)\otimes x+\lambda c, Q(t)\otimes y+\mu c]=P(t)Q(t)\otimes [x,y]+\langle x, y\rangle\,\mathrm{Res}_{t=0}\left(\frac{dP(t)}{dt}Q(t)\right)c.$$

Let $\widehat{\mathfrak{g}}$ be the affine Kac-Moody algebra obtained from $\widetilde{\mathfrak{g}}$ by adding a derivation $t\frac{d}{dt}$ which operates on $\CC[t, t^{-1}]\otimes \mathfrak{g}$ in an obvious way and sends $c$ to $0$.

The basic representation $V(\Lambda_0)$ is the $\widehat{\mathfrak{g}}$-module defined by the property that there is a nonzero vector $v_0$ (highest weight vector) in $V(\Lambda_0)$ such that $cv_0=v_0, (\CC[t]\oplus\CC t\frac{d}{dt})v_0=0$. Setting $V_k:=\{v\in V(\Lambda_0)| t\frac{d}{dt}=-kv\}$ gives a $\ZZ_+$-grading by finite spaces. Since $[g,d]=0$, each $V_k$ is a representation $\rho_k$ of $\mathfrak{g}$. Moreover, $\rho_1$ is the adjoint representation of $E_8.$

Let $q=e^{2\pi\ii \tau}, \tau\in \mathbb{H}$. Fix a basis for the Cartan subalgebra and let $\{z_i\}_{i=1}^8$ be the corresponding coordinates. The character of the basic representation is given by
$$\mathrm{ch}(z_1, z_2,\cdots,z_8,\tau):=\sum_{k=0}^{\infty}(\mathrm{ch}V_k)(z_1, z_2,\cdots, z_8)q^k=\varphi(\tau)^{-8}\Theta_{\mathfrak{g}}(z_1, z_2,\cdots, z_8, \tau),$$
where $\varphi(\tau)=\prod_{n=1}^\infty (1-q^n)$ so that $\eta(\tau)=q^{1/24}\varphi(\tau)$ is the Dedekind $\eta$ function; $\Theta_{\mathfrak{g}}(z_1, z_2,\cdots, z_8, \tau)$ is the theta function defined on the root lattice $Q$ by
$$\Theta_{\mathfrak{g}}(z_1, z_2,\cdots, z_8, \tau)=\sum_{\gamma\in Q}q^{|\gamma|^2/2}e^{2\pi\ii \gamma(\overrightarrow{z})},$$
where $\overrightarrow{z}=(z_1, z_2,\cdots, z_8)$.

It was proved in \cite{GL} that there is a basis for the $E_8$ root lattice such that
\be 
\Theta_{\mathfrak{g}}(z_1, \cdots. z_8,\tau)
=\frac{1}{2}\left( \prod_{l=1}^8\theta(z_l,\tau)+\prod_{l=1}^8\theta_1(z_l,\tau)+\prod_{l=1}^8\theta_2(z_l,\tau)+\prod_{l=1}^8\theta_3(z_l,\tau)\right),
\ee
where $\theta$ and $\theta_i$ ($i=1,\ 2,\ 3$) are the Jacobi theta functions (c.f. \cite{C} and \cite{HLZ}):
\be \theta(v,\tau)=2q^{1/8}\sin(\pi v)\prod_{j=1}^\infty[(1-q^j)(1-e^{2\pi \sqrt{-1}v}q^j)(1-e^{-2\pi
\sqrt{-1}v}q^j)], \ee
\be \theta_1(v,\tau)=2q^{1/8}\cos(\pi v)\prod_{j=1}^\infty[(1-q^j)(1+e^{2\pi \sqrt{-1}v}q^j)(1+e^{-2\pi
\sqrt{-1}v}q^j)], \ee
 \be \theta_2(v,\tau)=\prod_{j=1}^\infty[(1-q^j)(1-e^{2\pi \sqrt{-1}v}q^{j-1/2})(1-e^{-2\pi
\sqrt{-1}v}q^{j-1/2})], \ee
\be \theta_3(v,\tau)=\prod_{j=1}^\infty[(1-q^j)(1+e^{2\pi
\sqrt{-1}v}q^{j-1/2})(1+e^{-2\pi \sqrt{-1}v}q^{j-1/2})].\ee

The theta functions satisfy the the following
transformation laws (c.f. \cite{C}), 
\be \label{tran-0}
\theta(v,\tau+1)=e^{\pi \sqrt{-1}\over 4}\theta(v,\tau),\ \ \
\theta\left(v,-{1}/{\tau}\right)={1\over\sqrt{-1}}\left({\tau\over
\sqrt{-1}}\right)^{1/2} e^{\pi\sqrt{-1}\tau v^2}\theta\left(\tau
v,\tau\right)\ ;\ee 
\be \label{tran-1} \theta_1(v,\tau+1)=e^{\pi \sqrt{-1}\over
4}\theta_1(v,\tau),\ \ \
\theta_1\left(v,-{1}/{\tau}\right)=\left({\tau\over
\sqrt{-1}}\right)^{1/2} e^{\pi\sqrt{-1}\tau v^2}\theta_2(\tau
v,\tau)\ ;\ee 
\be \label{tran-2} \theta_2(v,\tau+1)=\theta_3(v,\tau),\ \ \
\theta_2\left(v,-{1}/{\tau}\right)=\left({\tau\over
\sqrt{-1}}\right)^{1/2} e^{\pi\sqrt{-1}\tau v^2}\theta_1(\tau
v,\tau)\ ;\ee 
\be \label{tran-3} \theta_3(v,\tau+1)=\theta_2(v,\tau),\ \ \
\theta_3\left(v,-{1}/{\tau}\right)=\left({\tau\over
\sqrt{-1}}\right)^{1/2} e^{\pi\sqrt{-1}\tau v^2}\theta_3(\tau
v,\tau)\ .\ee

\subsection{Proof of Theorem \ref{spinc-main} and Theorem \ref{spinc-new}} \label{proof-spinc}

Now we are ready to give the proofs. 

\subsubsection{Proof of Theorem \ref{spinc-main}} \label{proof of spinc-main} The proof of the statement (\ref{spinc-class}) about $p_c$ can be found in Theorem \ref{pc}. 

For the principal $E_8$ bundles $P_i, \ i=1,2,$ consider the associated bundles
$$\mathcal{V}_i=\sum_{k=0}^\infty \left(P_i\times_{\rho_k}V_k\right)q^k\in K(Z)[[q]].$$ Let $W_i=P_i\times_{\rho_1}V_1, i=1, 2,$ be the complex vector bundles associated to the adjoint representation $\rho_1$.

By the knowledge reviewed in Section \ref{affineE8}, we see that there are formally two forms $y^i_l$ ($i=1,2$ and $1\leq l \leq 8$) on $Z$ such that
\be \label{4theta} \varphi(\tau)^{8}\mathrm{ch}(\mathcal{V}_i)=\frac{1}{2}\left(\prod_{l=1}^8\theta(y^i_l,\tau)+\prod_{l=1}^8\theta_1(y^i_l,\tau)+\prod_{l=1}^8\theta_2(y^i_l,\tau)+\prod_{l=1}^8\theta_3(y^i_l,\tau)\right).\ee

Since $\theta(z,\tau)$ is an odd function about $z$, one can see that up to degree 12, the term $\prod_{l=1}^8\theta(y^i_l,\tau)$ can be dropped and therefore we have
\be \label{3theta-1}\varphi(\tau)^{8}\mathrm{ch}(\mathcal{V}_i)=\frac{1}{2}\left(\prod_{l=1}^8\theta_1(y^i_l,\tau)+\prod_{l=1}^8\theta_2(y^i_l,\tau)+\prod_{l=1}^8\theta_3(y^i_l,\tau)\right).\ee

Since $\theta_1(z,\tau), \theta_2(z,\tau)$ and $\theta_3(z,\tau)$ are all even functions about $z$, the right hand side of the above equality only contains even powers of $y^i_j$'s. Therefore $\mathrm{ch}(W_i)$ only consists of forms of degrees divisible by 4 (this is actually a basic fact about $E_8$). So
\be \label{chvi} \mathrm{ch}(\mathcal{V}_i)=1+\mathrm{ch}(W_i)q+\cdots=1+(248-c_2(W_i)+\cdots)q+\cdots
.\ee

On the other hand,
\be \label{3theta-2} \frac{1}{2}\left(\prod_{l=1}^8\theta_1(y^i_l,\tau)+\prod_{l=1}^8\theta_2(y^i_l,\tau)+\prod_{l=1}^8\theta_3(y^i_l,\tau)\right)=1+\left(240+30\sum_{l=1}^8(y^i_l)^2+\cdots\right)q+O(q^2).\ee

From (\ref{3theta-1}), (\ref{chvi}) and (\ref{3theta-2}), we have
\be \sum_{l=1}^8\left(y^i_l\right)^2=-\frac{1}{30}c_2(W_i). \ee

Let $TZ$ be the tangent bundle of $Z$ and $T_\CC Z$ be its complexification. Let $\xi$ be a rank two real oriented Euclidean vector bundle over
$Z$ carrying a Euclidean connection $\nabla^{\xi}$. Let $c=e(\xi,
\nabla^\xi)$ be the Euler form canonically associated to
$\nabla^\xi$. 

If $E$ is a complex vector bundle over $Z$, set
$\widetilde{E}=E-\mathbb{C}^{\mathrm{rk}(E)}. $ 
Recall that for an indeterminate $t$, \be \Lambda_t(E)=\CC
|_Z+tE+t^2\wedge^2(E)+\cdots,\ \ \ S_t(E)=\CC |_Z+tE+t^2
S^2(E)+\cdots, \ee are the total exterior and
symmetric powers of $E$ respectively. The following relations between these two operations hold (c.f. \cite{At}), 
\be S_t(E)=\frac{1}{\Lambda_{-t}(E)},\ \ \ \
\Lambda_t(E-F)=\frac{\Lambda_t(E)}{\Lambda_t(F)}.
\ee

Following \cite{CHZ}, set
\begin{equation*}
\Theta (T_{\CC}Z,\xi_\CC):= \left( \overset{
\infty }{\underset{m=1}{\otimes }}S_{q^{m}}(\widetilde{T_{\CC}Z}
)\right) \otimes \left( \overset{\infty }{\underset{n=1}{\otimes }}\Lambda
_{q^{n}}(\widetilde{\xi_\CC})\right)
\otimes \left( \overset{\infty }{\underset{u=1}{\otimes }}\Lambda
_{-q^{u-1/2}}(\widetilde{\xi_\CC})\right) \otimes
\left( \overset{\infty }{\underset{v=1}{\otimes }}\Lambda _{q^{v-1/2}}(
\widetilde{\xi_\CC})\right)\in K(Z)[[q]],
\end{equation*}
where $\xi_\CC$ is the complexification of $\xi$. 

Clearly, $\Theta(T_\CC Z, \xi_\CC)$ admits
a formal Fourier expansion in $q$ as \be \label{expand-B} \Theta(T_\CC Z,\xi_\CC)=\CC+B_1q+B_2q^2\cdots,\ee where
the $B_j$'s are elements in the semi-group formally generated by
complex vector bundles over $Z$. Moreover, they carry canonically
induced connections denoted by $\nabla^{B_j}$'s. Let
$\nabla^{\Theta}$ be the induced connection with
$q$-coefficients on $\Theta$.

For $1\leq i, j \leq 2$, set 
\be \label{QP1P2}
\begin{split}
&\mathcal{Q}(P_i, P_j, \xi, \tau)\\
:=&\left\{e^{\frac{1}{24}E_2(\tau)\left(p_1(TZ)-3c^2+\frac{1}{30}(c_2(W_i)+c_2(W_j))\right)}\widehat{A}(TZ)\cosh\left(\frac{c}{2}\right)\mathrm{ch}\left(\Theta(T_\CC Z, \xi_\CC)\right)\varphi(\tau)^{16}\mathrm{ch}(\mathcal{V}_i)\mathrm{ch}(\mathcal{V}_j)\right\}^{(12)}.
\end{split}
\ee
Here \be \label{expand-E2} E_2(\tau)=1-24\sum_{n=1}^{\infty}\left(\underset{d|n}{\sum}d\right)q^n\ee 
is the Eisenstein series. Unlike the other Eisenstein series $E_{2k}(\tau), k>1$, $E_2(\tau)$ is not a modular form over $SL(2,\ZZ)$, instead $E_2(\tau)$ is a quasimodular form over $SL(2,\ZZ)$, satisfying:
\be  E_2\left(\frac{a\tau+b}{c\tau+d}\right)=(c\tau+d)^2E_2(\tau)-\frac{6\sqrt{-1}c(c\tau+d)}{\pi}. \ee
In particular, we have
\be \label{tran-E2-1} E_2(\tau+1)=E_2(\tau),\ee
\be \label{tran-E2-2} E_2\left(-\frac{1}{\tau}\right)=\tau^2E_2(\tau)-\frac{6\sqrt{-1}\tau}{\pi}\ee
(c.f. Chap 2.3 in \cite{BGHZ}).

\begin{lemma}\label{Qxi} $\mathcal{Q}(P_i, P_j, \xi, \tau)$ is a modular form of weight 14 over $SL(2; \ZZ)$.
\end{lemma}

\begin{proof} Let $\{\pm 2\pi \ii x_k\} (1\leq k\leq 6)$ be the formal Chern roots for $(T_\CC Z, \nabla^{T_\CC Z})$. Let $c=e(\xi,
\nabla^\xi)=2\pi \ii u$. One has
\be
\begin{split}
&\mathcal{Q}(P_i, P_j, \xi, \tau)\\
=&\left\{e^{\frac{1}{24}E_2(\tau)\left(p_1(TZ)-3c^2+\frac{1}{30}(c_2(W_i)+c_2(W_j)\right)}\widehat{A}(TZ)\cosh\left(\frac{c}{2}\right)\mathrm{ch}\left(\Theta(T_\CC Z, \xi_\CC)\right)\varphi(\tau)^{16}\mathrm{ch}(\mathcal{V}_i)\mathrm{ch}(\mathcal{V}_j)\right\}^{(12)}\\
=&\left\{e^{\frac{1}{24}E_2(\tau)\left(p_1(TZ)-3c^2+\frac{1}{30}(c_2(W_i)+c_2(W_j)\right)}\left(\prod_{k=1}^{6}\left(x_k\frac{\theta'(0,\tau)}{\theta(x_k,\tau)}\right)
\right)\frac{\theta_1(u,\tau)}{\theta_1(0,\tau)}\frac{\theta_2(u,\tau)}{\theta_2(0,\tau)}\frac{\theta_3(u,\tau)}{\theta_3(0,\tau)}\right.\\
& \left.\cdot \frac{1}{4}\left(\prod_{l=1}^8\theta_1(y^i_l,\tau)+\prod_{l=1}^8\theta_2(y^i_l,\tau)+\prod_{l=1}^8\theta_3(y^i_l,\tau)\right)\left(\prod_{l=1}^8\theta_1(y^j_l,\tau)+\prod_{l=1}^8\theta_2(y^j_l,\tau)+\prod_{l=1}^8\theta_3(y^j_l,\tau)\right)\right\}^{(12)}.
\end{split}
\ee

Then we can perform the transformation laws (\ref{tran-0})-(\ref{tran-3}) for the theta functions and the transformation laws (\ref{tran-E2-1}), (\ref{tran-E2-2}) for $E_2(\tau)$ to show that $\mathcal{Q}(P_i, P_j, \xi, \tau)$ is a modular form of weight 14 over $SL(2; \ZZ)$.
\end{proof}

$\, $

Expanding the $q$-series, using (\ref{chvi}), (\ref{expand-B}) and (\ref{expand-E2}), we have
\be \label{expansion1}
\begin{split}
&e^{\frac{1}{24}E_2(\tau)\left(p_1(TZ)-3c^2+\frac{1}{30}(c_2(W_i)+c_2(W_j)\right)}\widehat{A}(TZ)\cosh\left(\frac{c}{2}\right)\mathrm{ch}\left(\Theta(T_\CC Z, \xi_\CC)\right)\varphi(\tau)^{16}\mathrm{ch}(\mathcal{V}_i)\mathrm{ch}(\mathcal{V}_j)\\
=&\left(e^{\frac{1}{24}\left(p_1(TZ)-3c^2+\frac{1}{30}(c_2(W_i)+c_2(W_j))\right)}\right.\\
&\left.\ \ \ \ -e^{\frac{1}{24}\left(p_1(TZ)-3c^2+\frac{1}{30}(c_2(W_i)+c_2(W_j))\right)}\left(p_1(TZ)-3c^2+\frac{1}{30}(c_2(W_i)+c_2(W_j))\right)q+O(q^2)\right)\\
&\cdot \widehat{A}(TZ)\cosh\left(\frac{c}{2}\right)\mathrm{ch}(\CC+B_1q+O(q^2))(1-16q+O(q^2))(1+\mathrm{ch}(W_i)q+O(q^2))(1+\mathrm{ch}(W_j)q+O(q^2))\\
=&e^{\frac{1}{24}\left(p_1(TZ)-3c^2+\frac{1}{30}(c_2(W_i)+c_2(W_j))\right)}\widehat{A}(TZ)\cosh\left(\frac{c}{2}\right)\\
&+q\left(e^{\frac{1}{24}\left(p_1(TZ)-3c^2+\frac{1}{30}(c_2(W_i)+c_2(W_j))\right)}\widehat{A}(TZ)\cosh\left(\frac{c}{2}\right)
\mathrm{ch}(B_1-16+W_i+W_j)\right.\\
&\ \ \ \left.\ \ \ -e^{\frac{1}{24}\left(p_1(TZ)-3c^2+\frac{1}{30}(c_2(W_i)+c_2(W_j))\right)}\left(p_1(TZ)-3c^2+\frac{1}{30}(c_2(W_i)+c_2(W_j))\right)\widehat{A}(TZ)\cosh\left(\frac{c}{2}\right)\right)\\
&+O(q^2).
\end{split}
\ee

It is well known (c.f. Chap 2.1 in \cite{BGHZ}) that modular forms over $SL(2; \ZZ)$ can be expressed as polynomials of the Eisenstein series $E_4(\tau)$, $E_6(\tau)$, where
\be E_4(\tau)=1+240q+2160q^2+6720q^3+\cdots,\ee
\be E_6(\tau)=1-504q-16632q^2-122976q^3+\cdots.\ee
Their weights are 4 and 6 respectively.

Since the weight of the modular form $\mathcal{Q}(P_i, P_j, \xi, \tau)$ is 14 and the space of modular forms of weight 14 over $SL(2,\mathbb{Z})$ is 1-dimensional and spanned by $E_4^2(\tau)E_6(\tau)$ (c.f. Chap 2.1 in \cite{BGHZ}), it must be a multiple of
\be \label{expansionE4^2E6} E_4(\tau)^2E_6(\tau)=1-24q+\cdots.\ee

So from (\ref{expansion1}) and (\ref{expansionE4^2E6}), we have
\be
\begin{split}
&\left\{e^{\frac{1}{24}\left(p_1(TZ)-3c^2+\frac{1}{30}(c_2(W_i)+c_2(W_j))\right)}\widehat{A}(TZ)\cosh\left(\frac{c}{2}\right)
\mathrm{ch}(B_1-16+W_i+W_j)\right\}^{(12)}\\
&\ \ \ \ \ \ -\left\{e^{\frac{1}{24}\left(p_1(TZ)-3c^2+\frac{1}{30}(c_2(W_i)+c_2(W_j))\right)}\left(p_1(TZ)-3c^2+\frac{1}{30}(c_2(W_i)+c_2(W_j))\right)\widehat{A}(TZ)\cosh\left(\frac{c}{2}\right)\right\}^{(12)}\\
=&-24\left\{e^{\frac{1}{24}\left(p_1(TZ)-3c^2+\frac{1}{30}(c_2(W_i)+c_2(W_j))\right)}\widehat{A}(TZ)\cosh\left(\frac{c}{2}\right)\right\}^{(12)}.
\end{split}
\ee

Therefore
\be \label{spinc-main-final}
\begin{split} &\left\{\widehat{A}(TZ)\cosh\left(\frac{c}{2}\right)\mathrm{ch}(W_i+W_j+B_1+8)\right\}^{(12)}\\
=&\left(p_1(TZ)-3c^2+\frac{1}{30}(c_2(W_i)+c_2(W_j)) \right)\\
&\cdot \left\{-\frac{e^{\frac{1}{24}\left(p_1(TZ)-3c^2+\frac{1}{30}(c_2(W_i)+c_2(W_j)\right)} -1}{p_1(TZ)-3c^2+\frac{1}{30}(c_2(W_i)+c_2(W_j))}\widehat{A}(TZ)\cosh\left(\frac{c}{2}\right)\mathrm{ch}(W_i+W_j+B_1+8) \right. \\
&\left.\ \ \ \ +e^{\frac{1}{24}\left(p_1(TZ)-3c^2+\frac{1}{30}(c_2(W_i)+c_2(W_j)\right)}\widehat{A}(TZ)\cosh\left(\frac{c}{2}\right)
\right\}^{(8)}
.\end{split}
\ee

To find $B_1$, we have
\be
\begin{split}
&\Theta (T_{\CC}Z,\xi_\CC)\\
=& \left( \overset{
\infty }{\underset{m=1}{\otimes }}S_{q^{m}}(\widetilde{T_{\CC}Z}
)\right) \otimes \left( \overset{\infty }{\underset{n=1}{\otimes }}\Lambda
_{q^{n}}(\widetilde{\xi_\CC})\right)
\otimes \left( \overset{\infty }{\underset{u=1}{\otimes }}\Lambda
_{-q^{u-1/2}}(\widetilde{\xi_\CC})\right) \otimes
\left( \overset{\infty }{\underset{v=1}{\otimes }}\Lambda _{q^{v-1/2}}(
\widetilde{\xi_\CC})\right)\\
=&1+(T_\CC Z-12-3\widetilde{\xi_\CC}-\widetilde{\xi_\CC}\otimes\widetilde{\xi_\CC})q+O(q^2).
\end{split}
\ee
So 
\be \label{B1}
B_1=T_\CC Z-12-3\widetilde{\xi_\CC}-\widetilde{\xi_\CC}\otimes\widetilde{\xi_\CC}.
\ee

Plugging $B_1$ into (\ref{spinc-main-final}), we have
\be \label{spinc-main-final1}
\begin{split} &\left\{\widehat{A}(TZ)\cosh\left(\frac{c}{2}\right)\mathrm{ch}(W_i+W_j+T_\CC Z-4-3\widetilde{\xi_\CC}-\widetilde{\xi_\CC}\otimes\widetilde{\xi_\CC} )\right\}^{(12)}\\
=&\left(p_1(TZ)-3c^2+\frac{1}{30}(c_2(W_i)+c_2(W_j)) \right)\\
&\cdot \left\{-\frac{e^{\frac{1}{24}\left(p_1(TZ)-3c^2+\frac{1}{30}(c_2(W_i)+c_2(W_j)\right)} -1}{p_1(TZ)-3c^2+\frac{1}{30}(c_2(W_i)+c_2(W_j))}\widehat{A}(TZ)\cosh\left(\frac{c}{2}\right)\mathrm{ch}(W_i+W_j+T_\CC Z-4-3\widetilde{\xi_\CC}-\widetilde{\xi_\CC}\otimes\widetilde{\xi_\CC} ) \right. \\
&\left.\ \ \ \ +e^{\frac{1}{24}\left(p_1(TZ)-3c^2+\frac{1}{30}(c_2(W_i)+c_2(W_j)\right)}\widehat{A}(TZ)\cosh\left(\frac{c}{2}\right)
\right\}^{(8)}
.\end{split}
\ee

Since $\mathrm{ch}(W_i), \mathrm{ch}(W_j)$ only contribute degree $4l$ forms, we can replace $\cosh\left(\frac{c}{2}\right)$ by $e^{\frac{c}{2}}$. Then in (\ref{spinc-main-final1}), putting $W_1=W_2=V_\CC(x)$, we get
 \be \label{spinc-form1}
\begin{split} &\left\{\widehat{A}(TZ)e^{\frac{c}{2}}\mathrm{ch}(V_\CC(x))+\frac{1}{2}\widehat{A}(TZ)e^{\frac{c}{2}}\mathrm{ch}(T_\CC Z)-\frac{1}{2}\widehat{A}(TZ)e^{\frac{c}{2}}\mathrm{ch}(4+3\widetilde{\xi_\CC}+\widetilde{\xi_\CC}\otimes\widetilde{\xi_\CC})\right\}^{(12)}\\
=&\left(\frac{p_1(TZ)-3c^2}{2}+\frac{1}{30}c_2(V_\CC(x)) \right)\\
&\cdot \left\{-\frac{e^{\frac{1}{24}\left(p_1(TZ)-3c^2+\frac{1}{15}c_2(V_\CC(x))\right)} -1}{p_1(TZ)-3c^2+\frac{1}{15}c_2(V_\CC(x))}\widehat{A}(TZ)e^{\frac{c}{2}}\mathrm{ch}(\mathfrak{A})  +e^{\frac{1}{24}\left(p_1(TZ)-3c^2+\frac{1}{15}c_2(V_\CC(x))\right)}\widehat{A}(TZ)e^{\frac{c}{2}}\right\}^{(8)}
,\end{split}
\ee
where $$\mathfrak{A}=2V_\CC(x)+T_\CC Z-4-3\widetilde{\xi_\CC}-\widetilde{\xi_\CC}\otimes\widetilde{\xi_\CC}.$$

It is not hard to check that $4+3\widetilde{\xi_\CC}+\widetilde{\xi_\CC}\otimes\widetilde{\xi_\CC}=\xi_\CC\otimes \xi_\CC-\xi_\CC+2$ and $$\frac{p_1(TZ)-3c^2}{2}+\frac{1}{30}c_2(V_\CC(x))=\lambda_c+2x=C_c(x).$$ A direct computation shows that the 8-form in the right hand side of (\ref{spinc-form1}) verifies 
\be 
\begin{split}
&\left\{-\frac{e^{\frac{1}{24}\left(p_1(TZ)-3c^2+\frac{1}{15}c_2(V_\CC(x))\right)} -1}{p_1(TZ)-3c^2+\frac{1}{15}c_2(V_\CC(x))}\widehat{A}(TZ)e^{\frac{c}{2}}\mathrm{ch}(\mathfrak{A})  +e^{\frac{1}{24}\left(p_1(TZ)-3c^2+\frac{1}{15}c_2(V_\CC(x)\right)}\widehat{A}(TZ)e^{\frac{c}{2}}\right\}^{(8)}\\
=&\frac{p_c-C_c(x)^2}{24}.
\end{split}
\ee
We therefore get (\ref{spinc-mainformula}), and have completed the proof of Theorem \ref{spinc-main}.

\subsubsection{Proof of Theorem \ref{spinc-new}} \label{proof of spinc-new} For each $i$, set
\be \label{RPi}
\begin{split} 
&\mathcal{R}(P_i, \xi, \tau)\\
:=&\left\{e^{\frac{1}{24}E_2(\tau)\left(p_1(TZ)-3c^2+\frac{1}{30}c_2(W_i)\right)}\widehat{A}(TZ)\cosh\left(\frac{c}{2}\right)\mathrm{ch}\left(\Theta(T_\CC Z, \xi_\CC)\right)\varphi(\tau)^{8}\mathrm{ch}(\mathcal{V}_i)\right\}^{(12)}.
\end{split}
\ee

\begin{lemma}$\mathcal{R}(P_i, \xi, \tau)$ is a modular form of weight 10 over $SL(2; \ZZ)$.
\end{lemma}
\begin{proof} This can be similarly proved as Lemma \ref{Qxi} by seeing that
\be
\begin{split}
&\mathcal{R}(P_i, \xi, \tau)\\
=&\left\{e^{\frac{1}{24}E_2(\tau)\left(p_1(TZ)-3c^2+\frac{1}{30}c_2(W_i)\right)}\widehat{A}(TZ)\cosh\left(\frac{c}{2}\right)\mathrm{ch}\left(\Theta(T_\CC Z, \xi_\CC)\right)\varphi(\tau)^{8}\mathrm{ch}(\mathcal{V}_i)\right\}^{(12)}\\
=&\left\{e^{\frac{1}{24}E_2(\tau)\left(p_1(TZ)-3c^2+\frac{1}{30}c_2(W_i)\right)}\left(\prod_{l=1}^{6}\left(x_l\frac{\theta'(0,\tau)}{\theta(x_l,\tau)}\right)
\right)\frac{\theta_1(u,\tau)}{\theta_1(0,\tau)}\frac{\theta_2(u,\tau)}{\theta_2(0,\tau)}\frac{\theta_3(u,\tau)}{\theta_3(0,\tau)}\right.\\
& \left.\cdot \frac{1}{2}\left(\prod_{l=1}^8\theta_1(y^i_l,\tau)+\prod_{l=1}^8\theta_2(y^i_l,\tau)+\prod_{l=1}^8\theta_3(y^i_l,\tau)\right)\right\}^{(12)},
\end{split}\ee
and then we can perform the transformation laws (\ref{tran-0})-(\ref{tran-3}) for the theta functions and the transformation laws (\ref{tran-E2-1}), (\ref{tran-E2-2}) for $E_2(\tau)$ to show that $\mathcal{R}(P_i, \xi, \tau)$ is a modular form of weight 10 over $SL(2; \ZZ)$.
\end{proof}

$\, $

Similar to the proof of Theorem \ref{spinc-main},
expanding the $q$-series, using (\ref{chvi}), (\ref{expand-B}) and (\ref{expand-E2}), we have
\be \label{expansion2}
\begin{split} 
&e^{\frac{1}{24}E_2(\tau)\left(p_1(TZ)-3c^2+\frac{1}{30}c_2(W_i)\right)}\widehat{A}(TZ)\cosh\left(\frac{c}{2}\right)\mathrm{ch}\left(\Theta(T_\CC Z, \xi_\CC)\right)\varphi(\tau)^{8}\mathrm{ch}(\mathcal{V}_i)\\
=&\left(e^{\frac{1}{24}\left(p_1(TZ)-3c^2+\frac{1}{30}c_2(W_i)\right)}\right.\\
&\left.\ \ \ \ -e^{\frac{1}{24}\left(p_1(TZ)-3c^2+\frac{1}{30}c_2(W_i)\right)}\left(p_1(TZ)-3c^2+\frac{1}{30}c_2(W_i)\right)q+O(q^2)\right)\\
&\cdot \widehat{A}(TZ)\cosh\left(\frac{c}{2}\right)\mathrm{ch}(\CC+B_1q+O(q^2))(1-8q+O(q^2))(1+\mathrm{ch}(W_i)q+O(q^2))\\
=&e^{\frac{1}{24}\left(p_1(TZ)-3c^2+\frac{1}{30}c_2(W_i)\right)}\widehat{A}(TZ)\cosh\left(\frac{c}{2}\right)\\
&+q\left(e^{\frac{1}{24}\left(p_1(TZ)-3c^2+\frac{1}{30}c_2(W_i)\right)}\widehat{A}(TZ)\cosh\left(\frac{c}{2}\right)
\mathrm{ch}(B_1-8+W_i)\right.\\
&\ \ \ \left.\ \ \ -e^{\frac{1}{24}\left(p_1(TZ)-3c^2+\frac{1}{30}c_2(W_i)\right)}\left(p_1(TZ)-3c^2+\frac{1}{30}c_2(W_i)\right)\widehat{A}(TZ)\cosh\left(\frac{c}{2}\right)\right)\\
&+O(q^2).
\end{split}
\ee

However modular form of weight 10 must be a multiple of (c.f. Chap 2.1 in \cite{BGHZ})
\be \label{expansionE4E6} E_4(\tau)E_6(\tau)=1-264q+\cdots.\ee
So from (\ref{expansion2}) and (\ref{expansionE4E6}), we have
\be
\begin{split}
&\left\{e^{\frac{1}{24}\left(p_1(TZ)-3c^2+\frac{1}{30}c_2(W_i)\right)}\widehat{A}(TZ)\cosh\left(\frac{c}{2}\right)
\mathrm{ch}(B_1-8+W_i)\right\}^{(12)}\\
&\ \ \ \ \ \ -\left\{e^{\frac{1}{24}\left(p_1(TZ)-3c^2+\frac{1}{30}c_2(W_i)\right)}\left(p_1(TZ)-3c^2+\frac{1}{30}c_2(W_i)\right)\widehat{A}(TZ)\cosh\left(\frac{c}{2}\right)\right\}^{(12)}\\
=&-264\left\{e^{\frac{1}{24}\left(p_1(TZ)-3c^2+\frac{1}{30}c_2(W_i)\right)}\widehat{A}(TZ)\cosh\left(\frac{c}{2}\right)\right\}^{(12)}.
\end{split}
\ee

Therefore
\be
\begin{split} &\left\{\widehat{A}(TZ)\cosh\left(\frac{c}{2}\right)\mathrm{ch}(W_i+B_1+256)\right\}^{(12)}\\
=&\left(p_1(TZ)-3c^2+\frac{1}{30}c_2(W_i) \right)\\
&\cdot \left\{-\frac{e^{\frac{1}{24}\left(p_1(TZ)-3c^2+\frac{1}{30}c_2(W_i)\right)} -1}{p_1(TZ)-3c^2+\frac{1}{30}c_2(W_i)}\widehat{A}(TZ)\cosh\left(\frac{c}{2}\right)\mathrm{ch}(W_i+B_1+256) \right. \\
&\left.\ \ \ \ +e^{\frac{1}{24}\left(p_1(TZ)-3c^2+\frac{1}{30}c_2(W_i)\right)}\widehat{A}(TZ)\cosh\left(\frac{c}{2}\right)
\right\}^{(8)}
.\end{split}
\ee

Plugging in $B_1$ (see (\ref{B1})), we have
\be
\begin{split} &\left\{\widehat{A}(TZ)\cosh\left(\frac{c}{2}\right)\mathrm{ch}(W_i+T_\CC Z+244-3\widetilde{\xi_\CC}-\widetilde{\xi_\CC}\otimes\widetilde{\xi_\CC} )\right\}^{(12)}\\
=&\left(p_1(TZ)-3c^2+\frac{1}{30}c_2(W_i) \right)\\
&\cdot \left\{-\frac{e^{\frac{1}{24}\left(p_1(TZ)-3c^2+\frac{1}{30}c_2(W_i)\right)} -1}{p_1(TZ)-3c^2+\frac{1}{30}c_2(W_i)}\widehat{A}(TZ)\cosh\left(\frac{c}{2}\right)\mathrm{ch}(W_i+T_\CC Z+244-3\widetilde{\xi_\CC}-\widetilde{\xi_\CC}\otimes\widetilde{\xi_\CC}) \right. \\
&\left.\ \ \ \ +e^{\frac{1}{24}\left(p_1(TZ)-3c^2+\frac{1}{30}c_2(W_i)\right)}\widehat{A}(TZ)\cosh\left(\frac{c}{2}\right)
\right\}^{(8)}
.\end{split}
\ee

Since $\mathrm{ch}(W_i)$ only contributes degree $4l$ forms, we can replace $\cosh\left(\frac{c}{2}\right)$ by $e^{\frac{c}{2}}$. Taking $W_i=V_\CC(x)$, we have
\be \label{spinc-form2}
\begin{split} &\left\{\widehat{A}(TZ)e^{\frac{c}{2}}\mathrm{ch}(V_\CC(x))+\widehat{A}(TZ)e^{\frac{c}{2}}\mathrm{ch}(T_\CC Z)+\widehat{A}(TZ)e^{\frac{c}{2}}\mathrm{ch}(244-3\widetilde{\xi_\CC}-\widetilde{\xi_\CC}\otimes\widetilde{\xi_\CC})\right\}^{(12)}\\
=&\left(p_1(TZ)-3c^2+\frac{1}{30}c_2(V_\CC(x)) \right)\\
&\cdot \left\{-\frac{e^{\frac{1}{24}\left(p_1(TZ)-3c^2+\frac{1}{30}c_2(V_\CC(x))\right)} -1}{p_1(TZ)-3c^2+\frac{1}{30}c_2(V_\CC(x))}\widehat{A}(TZ)e^{\frac{c}{2}}\mathrm{ch}(\mathfrak{B})  +e^{\frac{1}{24}\left(p_1(TZ)-3c^2+\frac{1}{30}c_2(V_\CC(x))\right)}\widehat{A}(TZ)e^{\frac{c}{2}}\right\}^{(8)}
,\end{split}
\ee
where $$\mathfrak{B}=V_\CC(x)+T_\CC Z+244-3\widetilde{\xi_\CC}-\widetilde{\xi_\CC}\otimes\widetilde{\xi_\CC}.$$

It is not hard to check that $244-3\widetilde{\xi_\CC}-\widetilde{\xi_\CC}\otimes\widetilde{\xi_\CC}=246-\xi_\CC\otimes \xi_\CC+\xi_\CC$ and $$p_1(TZ)-3c^2+\frac{1}{30}c_2(V_\CC(x))=2\lambda_c+2x=2\widetilde C_c(x).$$ A direct computation shows that the 8-form in the right hand side of (\ref{spinc-form2}) verifies 
\h
\begin{split}
&\left\{-\frac{e^{\frac{1}{24}\left(p_1(TZ)-3c^2+\frac{1}{30}c_2(V_\CC(x))\right)} -1}{p_1(TZ)-3c^2+\frac{1}{30}c_2(V_\CC(x))}\widehat{A}(TZ)e^{\frac{c}{2}}\mathrm{ch}(\mathfrak{E})  +e^{\frac{1}{24}\left(p_1(TZ)-3c^2+\frac{1}{30}c_2(V_\CC(x))\right)}\widehat{A}(TZ)e^{\frac{c}{2}}\right\}^{(8)}\\
=&\frac{\widetilde p_c+6\lambda_c\widetilde C_c(x)-4\widetilde C_c(x)^2}{24}.
\end{split}
\e
We therefore get (\ref{spinc-newformula}), and have completed the proof of Theorem \ref{spinc-new}.

\subsection{Proof of Theorem \ref{o1} and Theorem \ref{o2}} \label{proof-o1o2}

\subsubsection{Proof of Theorem \ref{o1}} Following \cite{CHZ}, set
\be \label{Theta1}
\Theta_1 (T_{\CC}Z):= \left( \overset{\infty }{\underset{n=1}{\otimes }}\Lambda
_{q^{n}}(\widetilde{T_{\CC}Z})\right)
\in K(Z)[[q]],
\ee
\be \label{Theta2}
\Theta_2 (T_{\CC}Z):= \left( \overset{\infty }{\underset{n=1}{\otimes }}\Lambda
_{-q^{u-1/2}}(\widetilde{T_{\CC}Z})\right)
 \in K(Z)[[q^{1/2}]],
\ee
\be \label{Theta3}
\Theta_3 (T_{\CC}Z):= \left( \overset{\infty }{\underset{n=1}{\otimes }}\Lambda
_{q^{u-1/2}}(\widetilde{T_{\CC}Z})\right)
\in K(Z)[[q^{1/2}]].
\ee
Construct (\cite{Liu95})
\be \Phi(T_\CC Z)=\Theta(T_\CC Z)\otimes \Theta_1 \left( T_{\mathbb{C}}Z\right)\otimes \Theta_2 \left( T_{\mathbb{C}}Z\right)\otimes\Theta_3 \left( T_{\mathbb{C}}Z\right)\in K(Z)[[q]]\ee
and define 
\be \label{L-Witten} \mathcal{W}_{\widehat{L}}(TZ)=e^{-\frac{1}{12}E_2(\tau)\cdot p_1(TZ)}\widehat{L}(TZ)\ch (\Phi(T_\CC Z))\in H^{4*}(TZ, \QQ).
 \ee
 We call $\mathcal{W}_{\widehat{L}}(TZ)$ the {\em $\widehat{L}$-Witten class} of $TZ$. 

Perform the formal Fourier expansion in $q$ as \be \label{expand-D} \Phi(T_\CC Z)=\Theta(T_\CC Z)\otimes \Theta_1 \left( T_{\mathbb{C}}Z\right)\otimes \Theta_2 \left( T_{\mathbb{C}}Z\right)\otimes\Theta_3 \left( T_{\mathbb{C}}Z\right)=\CC+D_1q+D_2q^2\cdots,\ee where
the $D_j$'s are elements in the semi-group formally generated by
complex vector bundles over $Z$. Moreover, they carry canonically
induced connections denoted by $\nabla^{D_j}$'s. Let
$\nabla^{\Phi}$ be the induced connection with
$q$-coefficients on $\Phi$.

For $1\leq i, j \leq 2$, construct the twisted $\widehat{L}$-Witten class 
\be \label{L-Witten-twofold} e^{\frac{1}{24}E_2(\tau)\left(-2p_1(TZ)+\frac{1}{30}(c_2(W_i)+c_2(W_j))\right)}\widehat{L}(TZ)\mathrm{ch}\left(\Phi(T_\CC Z)\right)\varphi(\tau)^{16}\mathrm{ch}(\mathcal{V}_i)\mathrm{ch}(\mathcal{V}_j)\in H^{4*}(Z, \QQ)\ee 
and denote \be
\begin{split}
&\mathcal{Q}_L(P_i, P_j, \tau)\\
=&\left\{e^{\frac{1}{24}E_2(\tau)\left(-2p_1(TZ)+\frac{1}{30}(c_2(W_i)+c_2(W_j))\right)}\widehat{L}(TZ)\mathrm{ch}\left(\Phi(T_\CC Z)\right)\varphi(\tau)^{16}\mathrm{ch}(\mathcal{V}_i)\mathrm{ch}(\mathcal{V}_j)\right\}^{(12)}.
\end{split}
\ee
\begin{lemma}\label{QL}$\mathcal{Q}_L(P_i, P_j, \tau)$ is a modular form of weight 14 over $SL(2; \ZZ)$.
\end{lemma}

\begin{proof}

Let $\{\pm 2\pi \ii x_k\} (1\leq k\leq 6)$ be the formal Chern roots for $(T_\CC Z, \nabla^{T_\CC Z})$. One has
\be
\begin{split}
&\mathcal{Q}_L(P_i, P_j,  \tau)\\
=&\left\{e^{\frac{1}{24}E_2(\tau)\left(-2p_1(TZ)+\frac{1}{30}(c_2(W_i)+c_2(W_j))\right)}\widehat{L}(TZ)\mathrm{ch}\left(\Phi(T_\CC Z)\right)\varphi(\tau)^{16}\mathrm{ch}(\mathcal{V}_i)\mathrm{ch}(\mathcal{V}_j)\right\}^{(12)}\\
=&2^6\left\{e^{\frac{1}{24}E_2(\tau)\left(-2p_1(TZ)+\frac{1}{30}(c_2(W_i)+c_2(W_j)\right)}\left(\prod_{k=1}^{6}x_k\frac{\theta'(0,\tau)}{\theta(x_k,\tau)}
\frac{\theta_1(x_k,\tau)}{\theta_1(0,\tau)}\frac{\theta_2(x_k,\tau)}{\theta_2(0,\tau)}\frac{\theta_3(x_k,\tau)}{\theta_3(0,\tau)}\right)\right.\\
& \left.\cdot \frac{1}{4}\left(\prod_{l=1}^8\theta_1(y^i_l,\tau)+\prod_{l=1}^8\theta_2(y^i_l,\tau)+\prod_{l=1}^8\theta_3(y^i_l,\tau)\right)\left(\prod_{l=1}^8\theta_1(y^j_l,\tau)+\prod_{l=1}^8\theta_2(y^j_l,\tau)+\prod_{l=1}^8\theta_3(y^j_l,\tau)\right)\right\}^{(12)}.
\end{split}
\ee

Then we can perform the transformation laws (\ref{tran-0})-(\ref{tran-3}) for the theta functions and the transformation laws (\ref{tran-E2-1}), (\ref{tran-E2-2}) for $E_2(\tau)$ to show that $\mathcal{Q}_L(P_i, P_j, \tau)$ is a modular form of weight 14 over $SL(2; \ZZ)$.
\end{proof}

$\, $

Expanding the $q$-series, using (\ref{chvi}), (\ref{expand-E2}) and (\ref{expand-D}), we have
\be \label{expansion3}
\begin{split}
&e^{\frac{1}{24}E_2(\tau)\left(-2p_1(TZ)+\frac{1}{30}(c_2(W_i)+c_2(W_j)\right)}\widehat{L}(TZ)\mathrm{ch}\left(\Phi(T_\CC Z)\right)\varphi(\tau)^{16}\mathrm{ch}(\mathcal{V}_i)\mathrm{ch}(\mathcal{V}_j)\\
=&\left(e^{\frac{1}{24}\left(-2p_1(TZ)+\frac{1}{30}(c_2(W_i)+c_2(W_j))\right)}\right.\\
&\left.\ \ \ \ -e^{\frac{1}{24}\left(-2p_1(TZ)+\frac{1}{30}(c_2(W_i)+c_2(W_j))\right)}\left(-2p_1(TZ)+\frac{1}{30}(c_2(W_i)+c_2(W_j))\right)q+O(q^2)\right)\\
&\cdot \widehat{L}(TZ)\mathrm{ch}(\CC+D_1q+O(q^2))(1-16q+O(q^2))(1+\mathrm{ch}(W_i)q+O(q^2))(1+\mathrm{ch}(W_j)q+O(q^2))\\
=&e^{\frac{1}{24}\left(-2p_1(TZ)+\frac{1}{30}(c_2(W_i)+c_2(W_j))\right)}\widehat{L}(TZ)\\
&+q\left(e^{\frac{1}{24}\left(-2p_1(TZ)+\frac{1}{30}(c_2(W_i)+c_2(W_j))\right)}\widehat{L}(TZ)
\mathrm{ch}(D_1-16+W_i+W_j)\right.\\
&\ \ \ \left.\ \ \ -e^{\frac{1}{24}\left(-2p_1(TZ)+\frac{1}{30}(c_2(W_i)+c_2(W_j))\right)}\left(-2p_1(TZ)+\frac{1}{30}(c_2(W_i)+c_2(W_j))\right)\widehat{L}(TZ)\right)+O(q^2).\\
\end{split}
\ee

Since the weight of the modular form $\mathcal{Q}(P_i, P_j, \tau)$ is 14, it must be a multiple of
\be \label{expansionE4^2E6-2} E_4(\tau)^2E_6(\tau)=1-24q+\cdots.\ee

So from (\ref{expansion3}) and (\ref{expansionE4^2E6-2}), we have
\be
\begin{split}
&\left\{e^{\frac{1}{24}\left(-2p_1(TZ)+\frac{1}{30}(c_2(W_i)+c_2(W_j))\right)}\widehat{L}(TZ)
\mathrm{ch}(D_1-16+W_i+W_j)\right\}^{(12)}\\
&\ \ \ \ \ \ -\left\{e^{\frac{1}{24}\left(-2p_1(TZ)+\frac{1}{30}(c_2(W_i)+c_2(W_j))\right)}\left(-2p_1(TZ)+\frac{1}{30}(c_2(W_i)+c_2(W_j))\right)\widehat{L}(TZ)\right\}^{(12)}\\
=&-24\left\{e^{\frac{1}{24}\left(-2p_1(TZ)+\frac{1}{30}(c_2(W_i)+c_2(W_j))\right)}\widehat{L}(TZ)\right\}^{(12)}.
\end{split}
\ee

Therefore
\be \label{o1-split}
\begin{split} &\left\{\widehat{L}(TZ)\mathrm{ch}(W_i+W_j+D_1+8)\right\}^{(12)}\\
=&\left(-2p_1(TZ)+\frac{1}{30}(c_2(W_i)+c_2(W_j)) \right)\\
&\cdot \left\{-\frac{e^{\frac{1}{24}\left(-2p_1(TZ)+\frac{1}{30}(c_2(W_i)+c_2(W_j)\right)} -1}{-2p_1(TZ)+\frac{1}{30}(c_2(W_i)+c_2(W_j))}\widehat{L}(TZ)\mathrm{ch}(W_i+W_j+D_1+8) \right. \\
&\left.\ \ \ \ +e^{\frac{1}{24}\left(-2p_1(TZ)+\frac{1}{30}(c_2(W_i)+c_2(W_j)\right)}\widehat{L}(TZ)\right\}^{(8)}
.\end{split}
\ee

To find $D_1$, we have
\be
\begin{split}
&\Phi(T_{\CC}Z)\\
=& \left( \overset{
\infty }{\underset{m=1}{\otimes }}S_{q^{m}}(\widetilde{T_{\CC}Z}
)\right) \otimes \left( \overset{\infty }{\underset{n=1}{\otimes }}\Lambda
_{q^{n}}(\widetilde{T_{\CC}Z})\right)
\otimes \left( \overset{\infty }{\underset{n=1}{\otimes }}\Lambda
_{-q^{u-1/2}}(\widetilde{T_{\CC}Z})\right)
\otimes
 \left( \overset{\infty }{\underset{n=1}{\otimes }}\Lambda
_{q^{u-1/2}}(\widetilde{T_{\CC}Z})\right)\\
=&1+(2T_\CC Z+\wedge^2(T_\CC Z)-S^2(T_\CC Z)-12)q+O(q^2).
\end{split}
\ee
So 
\be \label{D1} D_1=2T_\CC Z+\wedge^2(T_\CC Z)-S^2(T_\CC Z)-12.\ee

Plugging $D_1$ into (\ref{o1-split}), we have
\be 
\begin{split} &\left\{\widehat{L}(TZ)\mathrm{ch}(W_i+W_j+2T_\CC Z+\wedge^2(T_\CC Z)-S^2(T_\CC Z)-4)\right\}^{(12)}\\
=&\left(-2p_1(TZ)+\frac{1}{30}(c_2(W_i)+c_2(W_j)) \right)\\
&\cdot \left\{-\frac{e^{\frac{1}{24}\left(-2p_1(TZ)+\frac{1}{30}(c_2(W_i)+c_2(W_j)\right)} -1}{-2p_1(TZ)+\frac{1}{30}(c_2(W_i)+c_2(W_j))}\widehat{L}(TZ)\mathrm{ch}(W_i+W_j+2T_\CC Z+\wedge^2(T_\CC Z)-S^2(T_\CC Z)-4)+\right. \\
&\left.\ \ \ \ +e^{\frac{1}{24}\left(-2p_1(TZ)+\frac{1}{30}(c_2(W_i)+c_2(W_j)\right)}\widehat{L}(TZ)
\right\}^{(8)}
.\end{split}
\ee

Putting $W_1=W_2=V_\CC(x)$, we get
\be \label{orient-form1}
\begin{split}
&\left\langle 2\widehat L(TZ)\ch (V_\CC(x))+2\widehat L(TZ)\ch (T_\CC Z)+\widehat L(TZ)\ch(\wedge^2(T_\CC Z)-S^2(T_\CC Z))-4\widehat L(TZ), [Z] \right\rangle\\
=& \left(-2p_1(TZ)+\frac{1}{15}c_2(V_\CC(x)) \right)\\
&\cdot \left\{-\frac{e^{\frac{1}{24}\left(-2p_1(TZ)+\frac{1}{15}c_2(V_\CC(x)) \right)} -1}{-2p_1(TZ)+\frac{1}{15}c_2(V_\CC(x)) }\widehat{L}(TZ)\mathrm{ch}(\mathfrak{C})+e^{\frac{1}{24}\left(-2p_1(TZ)+\frac{1}{15}c_2(V_\CC(x)) \right)}\widehat{L}(TZ)
\right\}^{(8)},
\end{split}
\ee
where $$\mathfrak{C}=2V_\CC(x)+2T_\CC Z+\wedge^2(T_\CC Z)-S^2(T_\CC Z)-4.$$

Note that the 4-form $$-2p_1(TZ)+\frac{1}{15}c_2(V_\CC(x))=2D(x).$$ A direct computation shows that the 8-form in the right hand side of (\ref{orient-form1}) verifies 
\be 
\begin{split}
&\left\{-\frac{e^{\frac{1}{24}\left(-2p_1(TZ)+\frac{1}{15}c_2(V_\CC(x)) \right)} -1}{-2p_1(TZ)+\frac{1}{15}c_2(V_\CC(x)) }\widehat{L}(TZ)\mathrm{ch}(\mathfrak{C})+e^{\frac{1}{24}\left(-2p_1(TZ)+\frac{1}{15}c_2(V_\CC(x)) \right)}\widehat{L}(TZ)
\right\}^{(8)}\\
=&\frac{8}{3}\left(4p_1^2-7p_2-D(x)^2 \right).
\end{split}
\ee
We therefore get (\ref{o1-formula}), and have completed the proof of Theorem \ref{o1}.

\subsubsection{Proof of Theorem \ref{o2}}For each $i$, construct the twisted $\widehat{L}$-Witten class 
\be \label{L-Witten-single} e^{\frac{1}{24}E_2(\tau)\left(-2p_1(TZ)+\frac{1}{30}c_2(W_i)\right)}\widehat{L}(TZ)\mathrm{ch}\left(\Phi(T_\CC Z)\right)\varphi(\tau)^{8}\mathrm{ch}(\mathcal{V}_i)\in H^{4*}(Z, \QQ)\ee
and set
\be
\mathcal{R}_L(P_i, \tau)=\left\{e^{\frac{1}{24}E_2(\tau)\left(-2p_1(TZ)+\frac{1}{30}c_2(W_i)\right)}\widehat{L}(TZ)\mathrm{ch}\left(\Phi(T_\CC Z)\right)\varphi(\tau)^{8}\mathrm{ch}(\mathcal{V}_i)\right\}^{(12)}.
\ee

\begin{lemma}$\mathcal{R}_L(P_i, \tau)$ is a modular form of weight 10 over $SL(2; \ZZ)$.
\end{lemma}

\begin{proof}This can be similarly proved as Lemma \ref{QL} by seeing that
\be
\begin{split}
&\mathcal{R}_L(P_i, \tau)\\
=&\left\{e^{\frac{1}{24}E_2(\tau)\left(-2p_1(TZ)+\frac{1}{30}c_2(W_i)\right)}\widehat{L}(TZ)\mathrm{ch}\left(\Phi(T_\CC Z)\right)\varphi(\tau)^{8}\mathrm{ch}(\mathcal{V}_i)\right\}^{(12)}\\
=&2^6\left\{e^{\frac{1}{24}E_2(\tau)\left(-2p_1(TZ)+\frac{1}{30}c_2(W_i)\right)}\left(\prod_{k=1}^{6}x_k\frac{\theta'(0,\tau)}{\theta(x_k,\tau)}
\frac{\theta_1(x_k,\tau)}{\theta_1(0,\tau)}\frac{\theta_2(x_k,\tau)}{\theta_2(0,\tau)}\frac{\theta_3(x_k,\tau)}{\theta_3(0,\tau)}\right)\right.\\
& \left.\cdot \frac{1}{2}\left(\prod_{l=1}^8\theta_1(y^i_l,\tau)+\prod_{l=1}^8\theta_2(y^i_l,\tau)+\prod_{l=1}^8\theta_3(y^i_l,\tau)\right)\right\}^{(12)},
\end{split}\ee
and then we can perform the transformation laws (\ref{tran-0})-(\ref{tran-3}) for the theta functions and the transformation laws (\ref{tran-E2-1}), (\ref{tran-E2-2}) for $E_2(\tau)$ to show that $\mathcal{R}_L(P_i, \tau)$ is a modular form of weight 10 over $SL(2; \ZZ)$.
\end{proof}

$\, $

Expanding the $q$-series, using (\ref{chvi}), (\ref{expand-E2}) and (\ref{expand-D}), we have
\be \label{expansion4}
\begin{split}
&e^{\frac{1}{24}E_2(\tau)\left(-2p_1(TZ)+\frac{1}{30}c_2(W_i)\right)}\widehat{L}(TZ)\mathrm{ch}\left(\Phi(T_\CC Z)\right)\varphi(\tau)^{8}\mathrm{ch}(\mathcal{V}_i)\\
=&\left(e^{\frac{1}{24}\left(-2p_1(TZ)+\frac{1}{30}c_2(W_i)\right)}\right.\\
&\left.\ \ \ \ -e^{\frac{1}{24}\left(-2p_1(TZ)+\frac{1}{30}c_2(W_i)\right)}\left(-2p_1(TZ)+\frac{1}{30}c_2(W_i)\right)q+O(q^2)\right)\\
&\cdot \widehat{L}(TZ)\mathrm{ch}(\CC+D_1q+O(q^2))(1-8q+O(q^2))(1+\mathrm{ch}(W_i)q+O(q^2))\\
=&e^{\frac{1}{24}\left(-2p_1(TZ)+\frac{1}{30}c_2(W_i)\right)}\widehat{L}(TZ)\\
&+q\!\left(\!e^{\frac{1}{24}\left(-2p_1(TZ)+\frac{1}{30}c_2(W_i)\right)}\widehat{L}(TZ)
\mathrm{ch}(D_1\!\!-8+W_i)\!-\!e^{\frac{1}{24}\left(-2p_1(TZ)+\frac{1}{30}c_2(W_i)\right)}\left(\!-2p_1(TZ)\!+\!\frac{1}{30}c_2(W_i)\right)\widehat{L}(TZ)\right)\\
&+O(q^2).
\end{split}
\ee

However modular form of weight 10 must be a multiple of 
\be \label{expansionE4E6-2} E_4(\tau)E_6(\tau)=1-264q+\cdots. \ee So from (\ref{expansion4}) and (\ref{expansionE4E6-2})
 we have
\be
\begin{split}
&\left\{e^{\frac{1}{24}\left(-2p_1(TZ)+\frac{1}{30}c_2(W_i)\right)}\widehat{L}(TZ)
\mathrm{ch}(D_1-8+W_i)\right\}^{(12)}\\
&\ \ \ \ \ \ -\left\{e^{\frac{1}{24}\left(-2p_1(TZ)+\frac{1}{30}c_2(W_i)\right)}\left(-2p_1(TZ)+\frac{1}{30}c_2(W_i)\right)\widehat{L}(TZ)\right\}^{(12)}\\
=&-264\left\{e^{\frac{1}{24}\left(-2p_1(TZ)+\frac{1}{30}c_2(W_i)\right)}\widehat{L}(TZ)\right\}^{(12)}.
\end{split}
\ee

Therefore
\be
\begin{split} &\left\{\widehat{L}(TZ)\mathrm{ch}(W_i+D_1+256)\right\}^{(12)}\\
=&\left(-2p_1(TZ)+\frac{1}{30}c_2(W_i) \right)\\
&\cdot \left\{-\frac{e^{\frac{1}{24}\left(-2p_1(TZ)+\frac{1}{30}c_2(W_i)\right)} -1}{-2p_1(TZ)+\frac{1}{30}c_2(W_i)}\widehat{L}(TZ)\mathrm{ch}(W_i+D_1+256)+e^{\frac{1}{24}\left(-2p_1(TZ)+\frac{1}{30}c_2(W_i)\right)}\widehat{L}(TZ)\right\}^{(8)}
.\end{split}
\ee

Plugging in $D_1$ (see (\ref{D1})), we have
\be 
\begin{split} &\left\{\widehat{L}(TZ)\mathrm{ch}(W_i+2T_\CC Z+\wedge^2(T_\CC Z)-S^2(T_\CC Z)+244)\right\}^{(12)}\\
=&\left(-2p_1(TZ)+\frac{1}{30}c_2(W_i) \right)\\
&\cdot \left\{-\frac{e^{\frac{1}{24}\left(-2p_1(TZ)+\frac{1}{30}c_2(W_i)\right)} -1}{-2p_1(TZ)+\frac{1}{30}c_2(W_i)}\widehat{L}(TZ)\mathrm{ch}(W_i+2T_\CC Z+\wedge^2(T_\CC Z)-S^2(T_\CC Z)+244)\right. \\
&\left.\ \ \ \ +e^{\frac{1}{24}\left(-2p_1(TZ)+\frac{1}{30}c_2(W_i)\right)}\widehat{L}(TZ)
\right\}^{(8)}
.\end{split}
\ee

Taking $W_i=V_\CC(x)$, we have
\be \label{orient-form2}
\begin{split} &\left\{\widehat{L}(TZ)\mathrm{ch}(V_\CC(x)+2T_\CC Z+\wedge^2(T_\CC Z)-S^2(T_\CC Z)+244)\right\}^{(12)}\\
=&\left(-2p_1(TZ)+\frac{1}{30}c_2(V_\CC(x)) \right)\\
&\cdot \left\{-\frac{e^{\frac{1}{24}\left(-2p_1(TZ)+\frac{1}{30}c_2(V_\CC(x))\right)} -1}{-2p_1(TZ)+\frac{1}{30}c_2(V_\CC(x))}\widehat{L}(TZ)\mathrm{ch}(\mathfrak{D})+e^{\frac{1}{24}\left(-2p_1(TZ)+\frac{1}{30}c_2(V_\CC(x))\right)}\widehat{L}(TZ)
\right\}^{(8)},
\end{split}
\ee
where $$\mathfrak{D}=V_\CC(x)+2T_\CC Z+\wedge^2(T_\CC Z)-S^2(T_\CC Z)+244.$$ 

Note that the 4-form $$-2p_1(TZ)+\frac{1}{30}c_2(V_\CC(x))=2\widetilde D(x).$$ A direct computation shows that the 8-form in the right side of (\ref{orient-form2}) verifies
\h
\begin{split}
&\left\{-\frac{e^{\frac{1}{24}\left(-2p_1(TZ)+\frac{1}{30}c_2(V_\CC(x))\right)} -1}{-2p_1(TZ)+\frac{1}{30}c_2(V_\CC(x))}\widehat{L}(TZ)\mathrm{ch}(\mathfrak{D})+e^{\frac{1}{24}\left(-2p_1(TZ)+\frac{1}{30}c_2(V_\CC(x))\right)}\widehat{L}(TZ)
\right\}^{(8)}\\
=&\frac{8}{3}\left(p_1^2-7p_2-6p_1\widetilde D(x)-4\widetilde D(x)^2\right).
\end{split}
\e
We then get (\ref{o2-formula}), and have completed the proof of Theorem \ref{o2}.

\section{The characteristic classes in the cubic forms} \label{coeff}
The spin$^c$ characteristic classes are determined by Duan \cite{Duan18} by computing the integral cohomology of its classifying space $BSpin^{c}$. Let $c\in H^2(BSpin^{c})$ be the class with the mod $2$ reduction 
\[
c\equiv \omega_2~{\rm mod}~2,
\]
where $\omega_2$ is the second Stiefel-Whitney class. Then by a theorem of Duan \cite{Duan18} (c.f. Thomas \cite{Thomas62})
\begin{equation}\label{spinccohom}
H^\ast(BSpin^c)\cong \mathbb{Z}[c, q_1, q_2, q_3, \ldots ] \oplus ({\rm the}~ 2{\rm -torsion}~{\rm part}),
\end{equation}
where $q_i$ is called the {\it the $i$-th universal spin$^c$ class} with ${\rm deg}(q_i)=4i$. The spin$^c$ classes determine the Pontrjagin classes. In the low dimensions, we have
\be\label{pqforeq}
\begin{split}
&p_1=2q_1+c^2,\\
&p_2=2q_2+q_1^2,\\
&p_3=q_3.
\end{split}
\ee
The relations between spin$^c$ classes and Stiefel-Whitney classes can be described by the mod $2$ reductions of spin$^c$ classes (c.f. Benson-Wood \cite{BW95}). In the low dimensions,
\be\label{wqforeq}
\begin{split}
&q_1\equiv \omega_4 ~{\rm mod}~2,\\
&q_2\equiv \omega_8~{\rm mod}~2,\\
&q_3\equiv \omega_6^2~{\rm mod}~2.
\end{split}
\ee
To obtain spin characteristic classes, we can simply let $c=0$ in (\ref{spinccohom}) and (\ref{pqforeq}). In this case, (\ref{wqforeq}) is still valid.
Notice that in Freed-Hopkins \cite{FH19}, they denoted $q_1$ by $\lambda$ and $q_2$ by $p$.
With the above in hand, it is easy to calculate the following

\begin{theorem} \label{pc}
Let $M$ be any spin$^c$ manifold with the determinant class $c\in H^2(M)$. \newline
(i) One has
\[
4p_2-p_1^2-6p_1c^2+39c^4=8(q_2-2q_1c^2+4c^4).
\]
Hence,
\[
p_c=\frac{4p_2-p_1^2-6p_1c^2+39c^4}{8}
\]
is well defined and
\[\label{}
p_c\equiv \omega_8~{\rm mod}~2.
\]
(ii) One has
\[ 4p_2-7p_1^2+30p_1c^2-15c^4=8(q_2-3q_1^2+4q_1c^2+c^4).\\
\]
Hence
\[ 
\widetilde p_c=\frac{4p_2-7p_1^2+30p_1c^2-15c^4}{8}
\]
is well defined and 
\[\widetilde p_c\equiv \omega_8+\omega_4^2+\omega_2^4~{\rm mod}~2.\]
(iii) $p_c$ and $\widetilde p_c$ satisfy
\[ 
\widetilde p_c=p_c-3\lambda_c^2,
\]
where
\[
\lambda_c=\frac{1}{2}(p_1-3c^2)=q_1-c^2,
\]
and 
\[
\lambda_c\equiv \omega_4+\omega_2^2~{\rm mod}~2.
\]
\end{theorem}

\begin{theorem}
For any orientable manifold,
\[\label{}
\begin{split}
&4p_1^2-7p_2\equiv \omega_4^2 ~{\rm mod}~2,\\
&p_1^2-7p_2\equiv \omega_2^4+\omega_4^2 ~{\rm mod}~2.
\end{split}
\]
\end{theorem}

\section{Obstruction classes of spin$^{\xi}$ and spin$^{\omega_2}$ structures} \label{sectionspinw2}

\begin{definition}\label{spinxidef}
An oriented closed manifold $M$ is called {\bf spin$^{\xi}$} if its second Stiefel-Whitney class $w_{2}(M)$ can be realized as that of some real vector bundle $\xi$ of rank $2$ over $M$, that is,
\[
\omega_2(M)=\omega_2(\xi).
\]
\end{definition}
The concept spin$^{\xi}$ is a generalization of spin$^c$. Indeed, $M$ is spin$^c$ when $\xi$ can be chosen to be orientable. However, there are non-spin$^c$ spin$^{\xi}$ manifolds.
\begin{definition}\label{spinw2def}
An oriented closed manifold $M$ is called {\bf spin$^{\omega_2}$} if $\omega_{2}(M)$ can be realized as that of some nonorientable real vector bundle $\xi$ of rank $2$ over $M$.
\end{definition}

Recall that the obstruction class to spin$^c$ structure is the third integral Stiefel-Whitney class $W_{3}(M)\in H^3(M;\mathbb{Z})$. 
In contrast, we need to use \textit{cohomology with local coefficients} (or {\it twisted cohomology}) to investigate the obstructions of spin$^{\xi}$ and spin$^{\omega_2}$ structures. There are two standard ways to define cohomology with local coefficients: via module over the group ring of fundamental group, or via \textit{bundle of groups} (for instance, see Hatcher, Section $3.{\rm H}$ of \cite{Hatcher}). They correspond to each other in a natural way.

Here, we deal with the obstruction problem for the homotopy lifting diagram
\begin{gather}
\begin{aligned}
\xymatrix{
& BO(2)\ar[d]^{\omega_2}  \\
M \ar[r]_{ \omega_2(M)\ \ \ \  \ } \ar@{-->}[ru]^{f}  &K(\mathbb{Z}/2\mathbb{Z},2),
}
\end{aligned}
\label{spinw2liftdia}
\end{gather}
where $K(\mathbb{Z}/2\mathbb{Z},2)$ is Eilenberg-MacLane space, and $\omega_2$ represents the second universal Stiefel-Whitney class. 
For the classifying space $BO(2)$, the first universal Stiefel-Whitney class 
\[
\omega_1\in H^1(BO(2);\mathbb{Z}/2\mathbb{Z})\cong {\rm Hom}(\pi_1(BO(2)), {\rm Aut}(\mathbb{Z})\cong \mathbb{Z}/2\mathbb{Z})
\]
determines a bundle of groups $\mathbb{Z}^{\omega_1}\rightarrow BO(2)$ with fibre $\mathbb{Z}$. Moreover, $BO(2)$ is a \textit{generalized Eilenberg-MacLane space} in the sense of Gitler \cite{Gitler63}, and by Theorem $7.18$ of \cite{Gitler63}, for any $t\in H^1(M;\mathbb{Z}/2\mathbb{Z})\cong {\rm Hom}\big(\pi_1(M),\pi_1(BO(2))\big)$
\begin{equation}\label{trepeq}
[M, BO(2)]_{t}\cong H^2(M; \mathbb{Z}^{t}),
\end{equation}
where the set $[M, BO(2)]_{t}\subseteq [M, BO(2)]$ consists of the classes of maps $f$ such that $f_\ast=t$ on fundamental groups, and the local coefficient (or the bundle of groups) $\mathbb{Z}^{t}=t^\ast(\mathbb{Z}^{\omega_1})$.
From the short exact sequence of bundles of groups
\[
0\rightarrow \mathbb{Z}^{t}\stackrel{2}{\rightarrow}\mathbb{Z}^{t} \rightarrow \mathbb{Z}/2\mathbb{Z}\rightarrow 0,
\]
there exists a long exact sequence of cohomology with local coefficients
\begin{equation}\label{seqW3teq}
\cdots\rightarrow H^2(M; \mathbb{Z}^{t}) \stackrel{2}{\rightarrow} H^2(M; \mathbb{Z}^{t}) 
\stackrel{\rho_2}{\rightarrow} H^2(M; \mathbb{Z}/2\mathbb{Z}) \stackrel{\beta^t}{\rightarrow} H^3(M; \mathbb{Z}^{t})\rightarrow\cdots,
\end{equation}
where $\rho_2$ is the mod-$2$ reduction to the cohomology with the untwisted coefficient $\mathbb{Z}/2\mathbb{Z}$, and $\beta^t$ is the \textit{$t$-twisted Bockstein homomorphism}.
Let us call $W_3^{t}(M):=\beta^t(\omega_2(M))$ the \textit{third $t$-twisted integral Stiefel-Whitney class}. Set $M=BO$ and $t=\omega_1$, we have the \textit{third universal $t$-twisted integral Stiefel-Whitney class} 
$$W_3^{\omega_1}:=\beta^{\omega_1}(\omega_2).$$ In particular, for $t=0$ $W_3^{0}(M)=W_3(M)$ is the usual third integral Stiefel-Whitney class.
\begin{theorem}\label{spinxiobthm}
An oriented closed manifold $M$ is spin$^{\xi}$ if and only if
\[
W_3^{t}(M)=0,
\]
for some $t\in H^1(M;\mathbb{Z}/2\mathbb{Z})$.
\end{theorem}  
\begin{proof}
From the exactness of the sequence (\ref{seqW3teq}), $W_3^{t}(M)=0$ is equivalent to that $\omega_2(M)=\rho_2(c^t)$ for some $c^t\in  H^2(M; \mathbb{Z}^{t})$, which by (\ref{trepeq}) is equivalent to the existence of a real vector bundle $\xi^t$ of rank $2$ such that 
$\omega_1(\xi^t)=t$ and $\omega_2(\xi^t)=\omega_2(M)$. 
\end{proof}
In particular, the theorem recovers the obstruction result for spin$^c$ structure and determines the obstruction class for spin$^{\omega_2}$ structure as well.

\begin{corollary}\label{spinw2obthm}
An oriented closed manifold $M$ is spin$^{\omega_2}$ if and only if
\[
W_3^{t}(M)=0,
\]
for some $t\in H^1(M;\mathbb{Z}/2\mathbb{Z})$ and $t\neq 0$.
\end{corollary}

$$ $$

\end{document}